\documentclass{article}

\usepackage{microtype}
\usepackage{graphicx}
\usepackage{subfigure}
\usepackage{booktabs} 
\usepackage{array}
\usepackage[table]{xcolor}
\usepackage{amsmath}
\usepackage{amssymb}
\usepackage{amsthm}
\usepackage{graphicx}
\usepackage{amsbsy}
\usepackage{thmtools}
\usepackage{mathtools}
\usepackage{color}
\usepackage{caption}
\usepackage{hyperref}

\usepackage[accepted]{icml2021_modified}

\theoremstyle{plain}
\newtheorem{theorem}{Theorem}

\newtheorem{lemma}[theorem]{Lemma}
\newtheorem{proposition}[theorem]{Proposition}

\theoremstyle{definition}
\newtheorem{definition}[theorem]{Definition}

\theoremstyle{remark}

\newtheorem*{remark}{Remark}

\renewcommand\thmcontinues[1]{Continued}

\DeclareMathOperator*{\argmin}{argmin}
\DeclareMathOperator*{\tr}{tr}
\DeclareMathOperator{\interior}{int}

\newcommand{\Xd}{\mathbf{X}}
\newcommand{\Yd}{\mathbf{Y}}

\newcommand{\V}{\mathcal{V}}

\newcommand{\Id}{\mathrm{Id}}
\newcommand{\R}{\mathbb{R}}
\newcommand{\Nat}{\mathbb{N}}

\newcommand{\Prob}{\mathbb{P}}
\newcommand{\Qrob}{\mathbb{Q}}
\newcommand{\OL}{\mathcal{O}}

\newcommand{\eps}{\varepsilon}

\newcommand{\Ec}[2]{\textnormal{E}\left[#1\middle\vert#2\right]}

\newcommand{\E}[1]{\textnormal{E}\left[#1\right]}
\newcommand{\Esub}[2]{\textnormal{E}_{#1}\mkern-4mu\left[#2\right]}

\newcommand{\cov}[2]{\textnormal{cov}\left(#1,#2\right)}

\newcommand{\Shat}{\widehat{\Sigma}}

\newcommand{\N}{\mathcal{N}}

\newcommand{\bh}{\hat{\beta}}

\newcommand{\toD}{\xrightarrow{\text{D}}}

\newcommand{\e}{\mathrm{e}}
\newcommand{\Hc}{\mathcal{H}}
\newcommand{\MMD}{\textnormal{MMD}}
\newcommand{\HSIC}{\textnormal{HSIC}}
\newcommand{\one}{\mathrm{1}}
\newcommand{\Kt}{\tilde{K}}
\newcommand{\Lt}{\tilde{L}}
\newcommand{\HB}{\widehat{\textnormal{HSIC}}_\textnormal{b}}
\newcommand{\Hu}{\widehat{\textnormal{HSIC}}_\textnormal{u}}
\newcommand{\HBl}{\widehat{\textnormal{HSIC}}_\textnormal{block}}
\newcommand{\Hinc}{\widehat{\textnormal{HSIC}}_\textnormal{inc}}
\newcommand{\Hest}{\widehat{\textnormal{HSIC}}}
\newcommand{\HvBl}{H_{\textnormal{block}}}
\newcommand{\Hvinc}{H_{\textnormal{inc}}}
\newcommand{\Sk}{\mathcal{S}}
\newcommand{\D}{\mathcal{D}}
\newcommand{\sumP}{\sum\nolimits_{\mathmakebox[10pt][l]{j=1}}^{p}}

\newcommand{\bpar}{\hat{\beta}^{\mspace{1mu}\textnormal{par}}}

\newcommand{\Sh}{\hat{S}}
\newcommand{\St}{{\tilde{S}}}

\newcommand{\ns}{\cellcolor{white!100}}
\newcommand{\sel}{\cellcolor{lightgray!100}}
\newcommand{\ac}{\cellcolor{black!100}}

\newcommand\independent{\protect\mathpalette{\protect\independenT}{\perp}}
\def\independenT#1#2{\mathrel{\rlap{$#1#2$}\mkern2mu{#1#2}}}

\icmltitlerunning{Post-Selection Inference with HSIC-Lasso}

\begin{document}

\twocolumn[
\icmltitle{Post-Selection Inference with HSIC-Lasso}




\begin{icmlauthorlist}
\icmlauthor{Tobias Freidling}{cam}
\icmlauthor{Benjamin Poignard}{os,rik}
\icmlauthor{H\'ector Climente-Gonz\'alez}{rik}
\icmlauthor{Makoto Yamada}{rik,kyo}
\end{icmlauthorlist}

\icmlaffiliation{cam}{Department of Pure Mathematics and Mathematical Statistics, University of Cambridge, Cambridge, United Kingdom}
\icmlaffiliation{rik}{Center for Advanced Intelligence Project (AIP), RIKEN, Kyoto, Japan}
\icmlaffiliation{os}{Graduate School of Economics, Osaka University, Osaka, Japan}
\icmlaffiliation{kyo}{Graduate School of Informatics, Kyoto University, Kyoto, Japan}

\icmlcorrespondingauthor{Tobias Freidling}{taf40@cam.ac.uk}
\icmlkeywords{Post-selection inference, Kernel methods, Feature-selection}

\vskip 0.3in
]



\printAffiliationsAndNotice{}  


\begin{abstract}
Detecting influential features in non-linear and/or high-dimensional data is a challenging and increasingly important task in machine learning. Variable selection methods have thus been gaining much attention as well as post-selection inference. Indeed, the selected features can be significantly flawed when the selection procedure is not accounted for. We  propose a selective inference procedure using the so-called model-free "HSIC-Lasso" based on the framework of truncated Gaussians combined with the polyhedral lemma. We then develop an algorithm, which allows for low computational costs and provides a selection of the regularisation parameter. The performance of our method is illustrated by both artificial and real-world data based experiments, which emphasise a tight control of the type-I error, even for small sample sizes.
\end{abstract}

\section{Introduction}
The choice of a relevant statistical model in light of the observations prior to any statistical inference is ubiquitous in statistics. This reduces computational costs and fosters parsimonious models. For example, in linear regression analysis, allowing for a subset of features to enter the model, i.e. the sparsity assumption, is particularly suited to high-dimensional data, which potentially provides more robust predictions. Yet, should one follow the standard procedure, which assumes a model specified a priori, the data-driven nature of the selection procedure would be tacitly overlooked.
Therefore, applying classical inference methods may entail seriously flawed results, as it was highlighted by \citet{leeb-2005,leeb-2006}, among others.\newline\newline
The most straightforward remedy is sample splitting, which uses one part of the data for model selection and the other part for inference, cf. \cite{sample_splitting}. However, this approach leaves space for improvement as the selection-data is outright disregarded at the inference step. In particular, two major paradigms regarding the treatment of model selection have evolved. \citet{posi} developed a post-selection inference (PSI) method, which enables to circumvent distorting effects inherent to any selection procedure. Yet, in practice this method often leads to overly conservative confidence intervals and entails high computational costs.\newline
In contrast to this work, \citet{lee} and \citet{taylor14} only accounted for the actual selection outcome by conditioning on the latter at the inference step. This requires control over or at least insight into the selection procedure in order to characterise the selection event. Under the Gaussian assumption, it is possible to express this event as a restriction on the distribution of the inference target, using the so-called polyhedral lemma, and derive a pivotal quantity. Due to its comparatively low computational cost and general set-up, this method was applied with great success to a variety of model selection approaches, e.g. \cite{gen1} and \cite{l1_pen}.\newline
Variable selection procedure methods typically include step-wise procedures - see \cite{forward-stepwise} -, information criteria, with the AIC and its extensions - see \cite{AIC} -, and regularisation methods, pioneered by \citeauthor{lasso}'s work on the Lasso (\citeyear{lasso}). Their limitations, especially regarding assumptions such as linearity or certain probability distributions, fostered a flourishing research on model-free feature selection. The kernel-based Hilbert-Schmidt independence criterion (HSIC) (\citeyear{gretton}), proposed by \citeauthor{gretton}, emerged as a key ingredient as it enables to quantify the dependence between two random variables. This allows for simple feature selection: For instance, one can select a fixed number of covariates that exhibit the highest HSIC-estimates with the response, which is referred to as HSIC-ordering in the rest of the paper. HSIC-Lasso (\citeyear{hsic_lasso}), proposed by \citeauthor{hsic_lasso}, additionally accounts for the dependence structure among the covariates and selects features with an \(L^1\)-penalty. \citet{yamada2018} developed a selective inference approach for the former selection procedure; the analogous results for HSIC-Lasso, however, are still an important blind spot in PSI.\newline
In this study, we are interested in the following problem: Given $n$ observations of a response variable and potential features, select the relevant covariates using the HSIC-Lasso screening procedure; then, based on the polyhedral lemma, develop asymptotic inference \emph{given} the Lasso selection procedure in the same spirit as \citet{lee}. Our contributions are as follows: First, we derive a tailored, novel asymptotic pivot and characterise the selection event of HSIC-Lasso in an affine linear fashion; then we propose an algorithm that solves cumbersome issues arising in applications, such as high computational costs and hyper-parameter choice; finally, to illustrate our theoretical results and the relevance of the proposed method, we conduct an empirical analysis using artificial and real-world data. To the best of our knowledge, this work is the first approach to tackle the question of PSI with HSIC-Lasso.

\section{Background}
In this section the two theoretical cornerstones which our work is founded on - namely PSI based on truncated Gaussians and the Hilbert-Schmidt independence criterion - are reviewed.
\subsection{PSI with Truncated Gaussians}
We first review the PSI-approach (\citeyear{lee}), which was pioneered by \citeauthor{lee} and which considers the distribution of the quantity of interest, alias target, conditionally on a selection event at the inference step. We denote the set of potential models \(\Sk\), the model estimator \(\Sh\) and suppose that the response \(Y\) follows the distribution \(\N(\mu,\Sigma)\), where \(\mu\) is unknown and \(\Sigma\) is given. It is assumed that the inference target can be expressed as \(\eta_S^T\mu\) for a vector \(\eta_S\) that can depend on the previously chosen model \(S\). Consequently, the distribution of interest is \(\eta_S^T Y\vert\{\Sh = S\}\). If the selection event allows for an affine linear representation, i.e. \(\{\Sh = S\} = \{AY\leq b\}\), with \(A\) and \(b\) only depending on \(S\) and the covariates, a specific model choice can be seen as a restriction on the distribution of the inference target. This is the object of the following result.

\begin{lemma}[Polyhedral lemma]\label{lem:polyhedral_lemma}
        Let \(Y\sim \N(\mu, \Sigma)\) with \(\mu \in \R^n\) and \(\Sigma \in \R^{n\times n}\), \(\eta \in \R^n\), \(A \in \R^{k\times n}\) and \(b \in \R^k\). Defining \(Z := \left(\Id - C \eta^T\right)Y\) and \(C := (\eta^T\Sigma\eta)^{-1}\Sigma\eta\), then
        \begin{equation*}
                \left\{AY\leq b\right\} =
                \left\{\V^-(Z)\leq\eta^T Y\leq\V^+(Z)\right\},
        \end{equation*}
        holds almost surely with
        \begin{align}
        \V^-(Z) &:= \max_{j:(AC)_j<0}\frac{b_j-(AZ)_j}{(AC)_j},\label{V-}\\
        \V^+(Z) &:= \max_{j:(AC)_j>0}\frac{b_j-(AZ)_j}{(AC)_j}.\label{V+}
        \end{align}
\end{lemma}
Hence, selection restricts the values the inference target can take, which gives rise to the definition of truncated Gaussians.
\begin{definition}
        Let \(\mu\in\R\), \(\sigma^2 > 0\) and \(a, b \in\R\) such that \(a < b\). Then the cumulative distribution function of a \emph{Gaussian distribution \(\N(\mu,\sigma^2)\) truncated to the interval \([a,b]\)} is given by
        \begin{equation*}
        F_{\mu,\sigma^2}^{[a,b]}(x) = \frac{\Phi(\frac{x-\mu}{\sigma}) - \Phi(\frac{a-\mu}{\sigma})}{\Phi(\frac{b-\mu}{\sigma}) - \Phi(\frac{a-\mu}{\sigma})},
        \end{equation*}
        where \(\Phi\) denotes the cdf of \(\N(0,1)\).
\end{definition}
Concluding the line of thought, we are now able to state a pivotal quantity that can be used for inference.
\begin{theorem}\label{theo:polyhedral_lemma}
        Under the assumptions of Lemma \ref{lem:polyhedral_lemma} it holds that
        \begin{equation*}\label{eq:pivot}
        F_{\eta^T\mu,\eta^T\Sigma\eta}^{[\V^-(Z),\V^+(Z)]}(\eta^T Y)\vert \{A Y\leq b\} \sim\mathrm{Unif}\,(0,1),
        \end{equation*}
        where \(\V^-\) and \(\V^+\) are given by (\ref{V-}) and (\ref{V+}), respectively.
\end{theorem}
This result corresponds to Theorem 5.2 of \cite{lee}, which characterises the distribution of \(\eta^T Y\) conditionally on \(Y\) belonging to the polyhedron \(\{AY\leq b\}\).

\subsection{Hilbert-Schmidt Independence Criterion}
\citeauthor{gretton} proposed the Hilbert-Schmidt independence criterion (\citeyear{gretton}) as a model-free measure for the dependence between two random variables. It relies on embedding probability measures \(\Prob\) into a reproducing kernel Hilbert space \(\Hc_k\) with associated kernel function \(k\). If \(k\) is universal, i.e. \(\Hc_k\) is dense in the space of continuous functions, then the embedding \(\Prob \mapsto \mu_k(\Prob)\) is injective.
Hence, the squared distance between \(\mu_k(\Prob)\) and \(\mu_k(\Qrob)\) (maximum mean discrepancy) forms a metric so that \(\Prob = \Qrob \Leftrightarrow \lVert \mu_k(\Prob) - \mu_k(\Qrob)\rVert^2_{\Hc_k} \!\!= 0\) for any probability measures \(\Prob\) and \(\Qrob\); see \cite{HS_embed} for further details on RKHSs. The Gaussian kernel, which is universal, is probably the most commonly used kernel, and is defined as
\begin{equation*}
        k(x, y) = \exp\left(-\frac{\lVert x-y \rVert_2^2}{2\sigma^2}\right), \quad \sigma^2>0.
\end{equation*}
A way to detect the dependence between two random variables \(X\) and \(Y\) is the comparison between the joint distribution \(\Prob_{X,Y}\) and the product of the marginals \(\Prob_X \Prob_Y\) using the maximum mean discrepancy. This metric is precisely the Hilbert-Schmidt independence criterion, which can also be defined in terms of the involved kernels as follows.
\begin{definition}
        Let \(X\) and \(Y\) be random variables, \(X'\) and \(Y'\) independent copies, and \(k\) and \(l\) be bounded kernels. The \emph{Hilbert-Schmidt independence criterion} \(\HSIC_{k,l}\) is given by
        \begin{setlength}{\multlinegap}{0pt}
        \begin{multline*}
        \HSIC_{k,l}(X,Y) =\, \Esub{X,X',Y,Y'}{k(X, X')\;l(Y,Y')}\\
        \begin{aligned}
                &+ \Esub{X,X'}{k(X,X')} \Esub{Y,Y'}{l(Y,Y')}\\
                &- 2\, \Esub{X,Y}{\Esub{X'}{k(X,X')} \Esub{Y'}{l(Y,Y')}}.
        \end{aligned}
        \end{multline*}
        \end{setlength}
\end{definition}
Several approaches for estimating the Hilbert-Schmidt independence criterion from a data sample \(\{(x_j, y_j)\}_{j=1}^n\) have been put forward. \citet{gretton} proposed a V-statistic based estimator \(\HB(X,Y)\), which is biased, whereas \citet{song} provided an unbiased U-statistic version \(\Hu(X,Y)\). Regarding the asymptotic distribution, \citet{zhang} established that both these estimators scaled by \(n^{1/2}\) converge to a Gaussian distribution if \(X\) and \(Y\) are dependent. However for \(X \independent Y\), the asymptotic distribution is not normal. Since the true dependence between \(X\) and \(Y\) is unknown, we focus on estimators that are asymptotically normal in either case.\newline
In this respect, it is helpful to express the unbiased version of the HSIC-estimator as a U-statistic of degree 4 with kernel function \(h\) provided, e.g. in Theorem 3 in \citet{song}. \citet{zhang} and \citet{aistats} suggested the following estimators.
\begin{definition}\label{def:block_inc}
        Let \(B\in\Nat\) and subdivide the data into folds of size \(B\), \(\{\{(x_i^b,y_i^b)\}_{i=1}^B\}_{b=1}^{n/B}\). The \emph{block estimator} \(\HBl\) with block size \(B\) is given by
        \begin{equation}\label{eq:block}
        \HBl(X, Y) = \frac{1}{n/B}\sum_{b=1}^{n/B}\Hu(X^b,Y^b),
        \end{equation}
        where \(\Hu(X^b,Y^b)\) denotes the unbiased estimate on the data \(\{(x_i^b,y_i^b)\}_{i=1}^B\). Let \(\Sk_{n,4}\) be the set of all 4-subsets of \(\{1,\ldots,n\}\) and let \(\D\) be a multiset containing \(m\) elements of \(\Sk_{n,4}\) randomly chosen with replacement. Further, suppose \(m = \OL(n)\) and define \(l := \lim_{n,m\to\infty} m/n\). The \emph{incomplete U-statistic estimator} \(\Hinc\) of size \(l\) is defined by
        \begin{equation}\label{eq:inc}
        \Hinc(X,Y) = \frac{1}{m}\sum_{\mathmakebox[30pt][c]{(i,j,q,r)\in\D}} h(i,j,q,r).
        \end{equation}
\end{definition}
Both estimators are asymptotically normal.
\begin{theorem}\label{HSIC-CLT}
        Let \(\{(x^{(1)}_j,\dotsc,x^{(p)}_j,y_j)\}_{j=1}^n\) be an i.i.d. data sample and define \(H_0=(\HSIC(X^{(1)},Y),\dotsc\linebreak[1]\dotsc,\HSIC(X^{(p)},Y))^T\), and \(\HvBl\) and \(\Hvinc\) accordingly. Assume that \(B\) and \(l\) are the same for all entries of \(\HvBl\) and \(\Hvinc\), respectively, let \(n/B\to \infty\) and choose \(\D\) for all elements of \(\Hvinc\) independently. Then
        \begin{align*}
                \sqrt{n/B} \big(\HvBl - H_0\big) &\toD \N(0, \Sigma_\textnormal{block}),\\ 
                \sqrt{m} \big(\Hvinc - H_0\big) &\toD \N(0, \Sigma_\textnormal{inc}), 
        \end{align*}
        with positive definite matrices \(\Sigma_\textnormal{block}\) and \(\Sigma_\textnormal{inc}\).
\end{theorem}
The first statement is a direct consequence of the multidimensional central limit theorem; to prove the second asymptotic result, we use the framework of U-statistics. The details are deferred to the supplementary material.

\section{PSI for HSIC-Lasso}
This section contains our main theoretical results for PSI with HSIC-Lasso and addresses the difficulties arising in practical applications.
\subsection{Feature Selection with HSIC-Lasso}
\citeauthor{hsic_lasso} introduced the model-free HSIC-Lasso (\citeyear{hsic_lasso}) feature selection method as follows.
\begin{definition}\label{def:hsic_lasso}
        Let \(\{(x^{(1)}_j,\dotsc,x^{(p)}_j,y_j)\}_{j=1}^n\) be an i.i.d. data sample, \(k\) and \(l\) be kernels, let \(K\) and \(L\) be given by \(K_{ij} = k(x_i,x_j)\) and \(L_{ij} = l(y_i, y_j)\) for \(i,j\in\{1,\ldots,n\}\) and set \(\Gamma = \Id - \frac{1}{n}\one\one^T\). Denoting \(\bar{L} = \Gamma L \Gamma\) and \(\bar{K}^{(s)} = \Gamma K^{(s)} \Gamma\), \(\bh\) is given by
        \begin{equation*}
                \hat{\beta} = \argmin_{\beta\in\R^p_+} \;
                \frac{1}{2}\lVert \Bar{L}-\sum_{s=1}^p \beta_s\Bar{K}^{(s)}\rVert_{\text{Frob}}^2 + \lambda \lVert \beta\rVert_1,
        \end{equation*}
        where \(\lambda >0 \) and \(\R_+ = [0,\infty)\). The \emph{HSIC-Lasso selection procedure} is defined by \(\Sh=\{s\colon\bh_s > 0\}\).
\end{definition}
For our purposes, however, we consider the following alternative representation of \(\bh\):
\begin{multline}
        \hat{\beta} = \argmin_{\beta\in\R^p_+} \; -
        \sum_{s=1}^p\beta_s\, \HB(X^{(s)},Y)\\
        + \frac{1}{2} \sum_{s,r=1}^p \beta_s\beta_r\,\HB(X^{(s)},X^{(r)}) + \lambda \lVert \beta\rVert_1.\label{eq:hl_alt}
\end{multline}
It becomes apparent that feature selection is driven by three competing components. Considering the first and third term together, we notice that covariates with high dependence to the response achieve positive \(\bh\)-values, whereas the \(\bh\)-entries of non-influential features are forced to zero by the regularisation term. Moreover, the second term penalises 
the selection of covariates showing high dependence on other features.\newline
Taking (\ref{eq:hl_alt}) as a starting point, we replace the biased V-statistic based estimators with asymptotically normal ones and allow for a more general weighted Lasso-penalty.
\begin{definition}\label{def:normal_hsic_lasso}
        Let \(\textnormal{H}\) be an asymptotically Gaussian and \(\tilde{\textnormal{H}}\) be any HSIC-estimator and define \(H\) and \(M\) by \(H_s = \textnormal{H}(X^{(s)},Y)\), \(M_{sr} = \tilde{\textnormal{H}}(X^{(s)}, X^{(r)})\) for \(s,r\in\{1,\ldots,p\}\). The \emph{normal weighted HSIC-Lasso selection procedure} is given by \(\Sh=\{s\colon\bh_s > 0\}\) with
        \begin{equation}
                \hat{\beta} = \argmin_{\beta\in\R^p_+} \; -
                \beta^T H + \frac{1}{2}\,\beta^T M \beta + \lambda\, \beta^T\! w,
                \label{eq:wn_hsic-lasso}
        \end{equation}
        where \(w\in\R_+^p\) is a fixed weight vector.
\end{definition}
Using the asymptotically normal response \(H\), the framework of the polyhedral lemma can be applied; however, we need to provide an asymptotic pivot.
\begin{theorem}\label{th:asymp_pivot}
        Let \((H_n)_{n\in\Nat}\), \((M_n)_{n\in\Nat}\) and \((\Shat_n)_{n\in\Nat}\) be sequences of random vectors and matrices, respectively, such that \(H_n \to \N(\mu, \Sigma)\) in distribution, \(M_n\to M\), \(\Shat_n \to \Sigma\) almost surely and \(\Shat_n \independent (H_n, M_n)\). For a selected model \(S\) and \(\St \subseteq S\), let \(\eta_\St\), \(A_\St\) and \(b_\St\) be a.s. continuous functions of \(M\) and assume that the selection events are given by \(\{A_\St(M_n)H_n\leq b_\St(M_n)\}\) and that \(\interior(\{A_\St(M)H\leq b_\St(M)\}) \neq \emptyset\). Then
        \begin{equation}
                F_{\eta_n^T\mu,\eta_n^T\Shat_n\eta_n}^{[\V^-(Z_n),\V^+(Z_n)]}(\eta_n^T H_n)\vert \{A_n H_n\leq b_n\} \toD \mathrm{Unif}\,(0,1),
                \label{eq:asymp_pivot}
        \end{equation}
        as \(n\to\infty\) where \(\eta_n =\eta_\St(M_n)\), \(A_n = A_\St(M_n)\) and \(b_n = b_\St(M_n)\).
\end{theorem}
This statement is tailored to selection with normal weighted HSIC-Lasso and generalises Theorem~\ref{theo:polyhedral_lemma} as it relaxes the requirements of a normal response and known covariance \(\Sigma\): Under the assumption of an asymptotically Gaussian response and a consistent estimator for \(\Sigma\), we obtain an asymptotic pivot.\newline
To prove this result, we show that \((H_n, M_n, \Shat_n)\vert\{A_n H_n\leq b_n\} \to (H, M, \Sigma)\vert\{A H\leq b\}\) in distribution and then apply the continuous mapping theorem to the a.s. continuous function \(F\). The details of the proof can be found in the supplement.
\subsection{Inference Targets and Selection Event}
We now define the inference targets and characterise the selection events of the normal weighted HSIC-Lasso.
To do so, we first introduce the following notations. For any matrix \(B\in\R^{q\times q}\), \(v\in\R^q\) and index sets \(I, J\subset\{1,\ldots,q\}\), we define \(I^c := \{1,\ldots,q\}\setminus I\). Moreover, \(v_I\) contains all entries at positions in \(I\) and \(B_{IJ} \in\R^{\lvert I\rvert\times\lvert J\rvert}\) is given by the rows and columns of \(B\) whose indices are in \(I\) and \(J\), respectively.\newline
For a selected model \(S\) and \(j\in S\), we consider the HSIC-target \(H_j := \e_j^T H\) and the partial target \(\bpar_{j,S} := \e_j^T M_{SS}^{-1} H_S\), where \(\e_j\) denotes the \(j\)-th unit vector. The former target describes the dependence between response \(Y\) and feature \(X^{(j)}\). In the same spirit of a partial regression coefficient, the latter can be interpreted as the degree of influence of \(X^{(j)}\) on \(Y\) adjusted to the dependence structure among the covariates. Expressing both targets in the form of \(\eta_S^T H\), the respective \(\eta\)-vectors for the HSIC- and partial target are \(\e_j\) and \((M_{SS}^{-1}\,\vert\,0)^T\e_j\).\newline
We notice that the HSIC-target is influenced by the selection information \(\{j\in\Sh\}\) only, whereas the partial target is, by definition, affected by the entire chosen set of covariates \(\{\Sh = S\}\). We characterise these selection events as follows.
\begin{theorem}\label{th:trunc_points}
        Assume the same framework as in Definition \ref{def:normal_hsic_lasso}, suppose that \(M\) is positive definite and let \(\eta\in\R^p\). Then \({\{\hat{S} = S\}} =\{A\, (H_S, H_{S^c})^T \leq b\}\) holds, where
        \begin{equation}\label{eq:aff_lin_partial}
                A = -\frac{1}{\lambda}\left(
                \begin{smallmatrix}
                M_{SS}^{-1} &\vert& 0\,\\
                M_{S^cS}M_{SS}^{-1}&\vert& \Id\,
                \end{smallmatrix}\right),\quad
                b = \left(
                \begin{smallmatrix}
                -M_{SS}^{-1}\,w_S\\
                w_{S^c}-M_{S^cS}M_{SS}^{-1}\,w_S
                \end{smallmatrix}\right),
        \end{equation}
        and \(0\) denotes a matrix of size \(\lvert S\rvert\times\lvert S^c\rvert\) filled with zeros. Moreover, \(\{j\in \hat{S}\}=\{A H\leq b\}\) holds for
        \begin{equation}\label{eq:aff_lin_H}
            A = -\,\e_j^T,\quad
            b = -\,\e_j^T M \bh_{-j}\! - \lambda w_j,
        \end{equation}
        where \(\bh_{-j}\!\) denotes \(\bh\) with the \(j\)-th entry set to zero.
\end{theorem}

To prove these statements, we use the Karush-Kuhn-Tucker (KKT) conditions, which characterise the solution of (\ref{eq:wn_hsic-lasso}) by a set of inequalities. Manipulating these, we obtain the affine linear representation of \({\{\hat{S} = S\}}\) and \(\{j\in \hat{S}\}\). The details can be found in the supplementary material.\newline
Theorem \ref{th:trunc_points} is the key result that allows us to carry out post-selection inference with HSIC-Lasso. In summary, we have to consider the distributions \({\e_j^T H\,\vert\,\{j\in\Sh\}}\) and
\(\e_j^T (M_{SS}^{-1}H\,\vert\,0)\,\vert\,\{\Sh = S\}\) for the HSIC- and partial target, respectively, in order to account for the selection. With the affine linear representations (\ref{eq:aff_lin_partial}) and (\ref{eq:aff_lin_H}), we can apply Theorem \ref{th:asymp_pivot} and get an asymptotic pivot for inference.
\subsection{Practical Applications}
Equipped with these theoretical results, we now propose an algorithm that handles difficulties arising in practical applications.

\subsubsection{Positive Definiteness}
Theorem \ref{th:trunc_points} requires \(M\) to be positive definite. For the original version of HSIC-Lasso (\ref{eq:hl_alt}), this condition is always fulfilled by the structure of the biased HSIC-estimates. However, there is no guarantee for other estimation procedures. For this reason, we project \(M\) onto the space of positive definite matrices,
as proposed by \citet{higham}: The spectral decomposition of \(M\) is computed and all negative eigenvalues are replaced with a small positive value \(\eps > 0\).

\subsubsection{Computational Costs}
HSIC-Lasso is frequently applied to high-dimensional data where the number of covariates \(p\) exceeds the sample size \(n\). The resulting high computational costs are caused by the calculation of the HSIC-estimates, where \(H\) grows as \(\OL(p)\) and \(M\) as \(\OL(p^2)\). Therefore, we introduce an upstream screening stage identifying a subset of potentially influential features so that HSIC-Lasso only has to deal with these. Following the approach of \citet{yamada2018}, we compute the HSIC-estimates \(\Hest(Y,X^{(j)}), j\in\{1,\ldots,p\}\), and select a pre-fixed number \(p' < p\) of the covariates having the highest estimates. We call this HSIC-ordering. 
\newline
In order to ensure valid inference results, we have to adjust for the screening step as well because it affects feature selection. To do so, we split the data into two folds, one dedicated to screening, the other dedicated to HSIC-Lasso selection among the screened variables, cf. \cite{sample_splitting}. Thus, potentially distorting effects of the screening step are separated from inference on the second fold.
Moreover, unbiased HSIC-estimates can be used for screening, which are more precise than block or incomplete U-statistic estimates.
\begin{remark}
    In future applications, random Fourier features could be used to speed up the kernel computation of the objective function (\ref{eq:hl_alt}), allowing for a larger $p'$. However, it is not immediately clear whether we can recover the theoretical guarantees for our method when using approximated kernel functions.
\end{remark}

\subsubsection{Hyper-parameter Choice}
In practice, a suitable choice of the regularisation parameter \(\lambda\) and the weight vector \(w\) is key for meaningful results. Since the data generating process is unknown, we have to estimate these hyper-parameters. In order to prevent this from affecting inference results, we use the first fold for hyper-parameter selection.
In doing so, we can apply any estimation method for \(\lambda\), such as cross-validation, e.g. \cite{cross_val}, or the AIC, cf. \cite{AIC}, and get a valid procedure that is easy to implement. Moreover, we can employ \citeauthor{adaptive_lasso}'s adaptive Lasso penalty (\citeyear{adaptive_lasso}), that uses the weight vector \(w = 1 / \lvert\hat{\beta} \rvert^\gamma\). \(\gamma\) is typically set to \(1.5\) or \(2\) and \(\hat{\beta}\) is a \(\sqrt{n}\,\)-consistent estimator, e.g. the ordinary least squares estimator. Contrary to the vanilla Lasso, this method satisfies the oracle property, that is the sparsity-based estimator recovers the true underlying sparse model and has an asymptotically normal distribution. Yet, this property was only proven for a covariance matrix of the form \(\Sigma = \sigma^2\Id\). Hence, we have to evaluate the usefulness of the adaptive Lasso in empirical simulations.

\subsubsection{Algorithm}\label{sec:algorithm}
We summarise our proposed PSI method for a normal (weighted) HSIC-Lasso selection procedure as follows. To begin with, we split the data into two subsets. On the first fold, we compute the HSIC-estimates between all covariates and the response, determine the most influential \(p'\) features (screening) and estimate the hyper-parameters \(\lambda\) and \(w\) with the methods previously specified. On the second fold, we compute the estimates of \(H\), \(M\) and \(\Shat\) for the screened features and solve the optimisation problem (\ref{eq:wn_hsic-lasso}). For all \(j\in\{1,\ldots,p\}\) such that \(\bh_j > 0\), we find the truncation points for the specified targets with Theorem \ref{th:trunc_points} and Lemma \ref{lem:polyhedral_lemma} and test with the asymptotic pivot (\ref{eq:asymp_pivot}) if the targets are significant at a given level \(\alpha\). The supplementary material contains a more detailed description of the algorithm in pseudocode.
\begin{remark}
        Although Theorem \ref{th:asymp_pivot} requires independence between the covariance- and the \((H,M)\)-estimate we could not observe any detrimental influence if \(\Shat\) is computed as outlined above.
\end{remark}


\section{Experiments}
In this section, we illustrate our theoretical contribution and the proposed algorithm on artificial and real-world data. The source code for the experiments was implemented in Python and relies on \citeauthor{aistats}'s \texttt{mskernel}-package (\citeyear{aistats}). We use the Lasso optimisation routines of \texttt{scikit-learn} which implements the cyclical gradient descent algorithm, cf. \cite{coordinate_descent}, and the least angle regression algorithm (LARS) (\citeyear{lar}), proposed by \citeauthor{lar}. The source code for the following experiments is available on \href{https://github.com/tobias-freidling/hsic-lasso-psi}{Github}: \texttt{tobias-freidling/hsic-lasso-psi}.

\subsection{Artificial Data}
We examine the achieved type-I error of the proposed algorithm, compare its power to other approaches for post-selection inference and briefly discuss additional experiments.\newline
For continuous data, we use Gaussian kernels where the bandwidth parameter is chosen according to the median heuristic, cf. \cite{median-heuristic}; for discrete data with $n_c$ samples in category $c$, we apply the normalised delta kernel which is given by
\begin{equation*}
l(y, y') :=
        \begin{cases}
                1/n_c, &\text{if }\, y=y'\!=c,\\
                0, &\text{otherwise}.
        \end{cases}
\end{equation*}
Moreover, we use a quarter of the data for the first fold, select the hyper-parameter $\lambda$ applying 10-fold cross-validation with MSE, use a non-adaptive Lasso-penalty and do not conduct screening as the number of considered features is already small enough. On the second fold, we estimate \(M\) with the block estimator of size \(B=10\) as it is computationally less expensive than the unbiased estimator and leads to similar results. The covariance matrix \(\Sigma\) of \(H\) is estimated based on the summands of the block (\ref{eq:block}) and incomplete U-statistic (\ref{eq:inc}) estimator, respectively. To this end, we use the oracle approximating shrinkage (OAS) estimator (\citeyear{oas}), which was presented by \citeauthor{oas} and is particularly tailored for high-dimensional Gaussian data. We fix the significance level at \(\alpha=0.05\) and simulate 100 datasets for each considered sample size.

\subsubsection{Type-I Error}
In order to simulate the achieved type-I error we use the toy models
\begin{align*}
    &\text{(M1)}\quad
    \begin{gathered}
        Y\sim \textnormal{Ber}\Big(g\big(\textstyle\sum_{i=1}^{10} X_i\big)\Big),\quad  X \sim \N(0_{50},\Xi),\\
        g(x) = \e^x/(1+\e^x),
    \end{gathered}\\[0.3em]
    &\text{(M2)}\quad
    \begin{gathered}
        Y = \textstyle\sum_{i=1}^5 X_i X_{i+5} + \eps, \quad X \sim \N(0_{50},\Xi),\\
        \eps\sim\N(0,\sigma^2),
    \end{gathered}
\end{align*}
where \(0_{50} \in\R^{50}\) and \(\Xi\in\R^{50\times 50}\), to generate the data. These are clearly non-linear and cover categorical and continuous responses. In (M2), we choose \(\sigma^2\) such that the variance of \(\eps\) is a fifth of the variance of the \(X\)-dependent terms of \(Y\) amounting to a noise-to-signal ratio of \(0.2\). As for the covariance matrix \(\Xi\), two cases are considered: we either set \(\Xi = \Id\) or use decaying correlation, i.e. \(\Xi_{ij} = 0.5^{\lvert i-j\rvert}\).\newline
We simulate datasets with sample sizes $n\in\{400, 800, 1200, 1600\}$ for all different settings of models and covariance matrices and estimate $H$ with block estimators of sizes 5 and 10 as well as with an incomplete U-statistics estimator of size $l=1$. Since the partial target both depends on the entire set of selected variables and its value cannot be directly inferred from the data-generating mechanism, it is inherently hard to rigorously assess the type-I error of any given partial target. (However, the false positive rate for testing different partial targets hints that the type-I error is probably close to 0.05). For this reason, we concentrate on the HSIC-target in our analysis. For both models $\HSIC(Y,X^{(j)}) = 0, j\in\{11,\ldots,50\},$ holds which allows us to estimate the type-I error as the ratio of null hypothesis rejections and tests among the selected features with indices in $\{11,\ldots,50\}$. If influential variables are correlated with uninfluential ones, the HSIC-value is not precisely zero; nonetheless, the use of \emph{decaying} correlation renders the bias of this effect ignorable. Figure \ref{fig:typeI-error} illustrates that the type-I error is close to 0.05 across all estimators and data generating mechanisms, even for small sample sizes.
\begin{figure*}[!ht]
    \vskip 0.1in
    \begin{center}
    \centerline{\includegraphics[width=17cm, height=3.5cm]{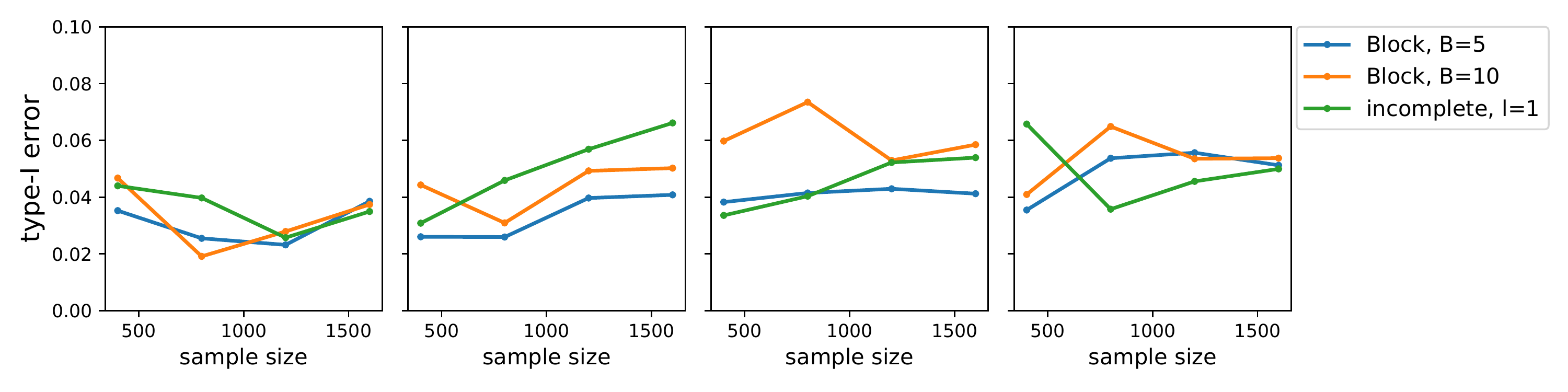}}
    \caption{Type-I error for the HSIC-target in different toy models (from left to right): (M1) with $\Xi = \Id$, (M1) with $\Xi_{ij} = 0.5^{\lvert i-j\rvert}$,\\ (M2) with $\Xi = \Id$, (M2) with $\Xi_{ij} = 0.5^{\lvert i-j\rvert}$}
    \label{fig:typeI-error}
    \end{center}
    \vskip -0.2in
\end{figure*}

\subsubsection{Power}
In this set of experiments, we adapt the toy model (M1), replacing $X_1$ by $\theta X_1$ and setting $\Xi = \Id$, and denote it (M1'). Moreover, we introduce the following linear and modified linear model
\begin{align*}
    &\text{(M3)}\quad
    \begin{gathered}
        Y = \theta X_1 + \textstyle\sum_{i=2}^{10} X_i + \eps, \quad X \sim \N(0_{50},\Id),\\
        \eps\sim\N(0,\sigma^2),
    \end{gathered}\\[0.3em]
    &\text{(M4)}\quad
    \begin{gathered}
        Y = \theta\, h(X_1) + \textstyle\sum_{i=2}^{10} X_i + \eps, \quad X \sim\N(0_{50},\Id),\\
        h(x) = x-x^3,\quad\eps\sim\N(0,\sigma^2),
    \end{gathered}
\end{align*}
where $\theta\in\R$. Our proposed algorithm is applied with both a block estimator, $B=10$, and an incomplete U-statistics estimator, $l=1$, and is compared with the so-called Multi PSI-approach, presented by \citet{aistats}: We select \(k\) features with HSIC-ordering and carry out inference for the HSIC-targets. (Since \(M\) is not involved in the feature selection, we cannot define partial targets for HSIC-ordering.) It was empirically shown that multiscale bootstrapping (\citeyear{shim2004}), which was first presented by \citeauthor{shim2004} and is abbreviated by Multi, is a more powerful PSI-approach for HSIC-ordering than truncated Gaussians. In our simulations, we set $k=15$ and applied Multi with a block estimator, $B=10$, as well as an incomplete U-statistics estimator, $l=1$. Additionally, we applied \citeauthor{lee}'s original PSI-method (\citeyear{lee}), that relies on Lasso-regularisation and assumes a linear regression setting, to (M3) and (M4). The inference target in this case is the partial regression coefficient.\newline
We simulate datasets for values of $\theta$ in $\{0.00, 0.33, 0.67, 1.00, 1.33, 1.67, 2.00, 2.33\}$ and a sample size of $n=800$, and compute the ratio of rejections of the null-hypothesis, i.e. the respective inference target corresponding to $X_1$ is zero, and the number of tests carried out. Plotting the obtained ratios against $\theta$ does not correspond to the usual depiction of the power function as $\theta$ is not the inference target for all considered procedures. However, this allows for an intuitive understanding of how strong $X_1$ needs to influence $Y$ in order to be detected.\newline
Figure \ref{fig:power} exhibits that the power of our proposed algorithm is similar to the Multi procedure, especially when using the block estimator. This confirms that, even without a manual choice of the number of selected features and costly bootstrap sampling, it is possible to match the best performing model-free PSI-methods. Moreover, we observe that regularised linear regression clearly outperforms our procedure as well as Multi for small values of $\theta$ if the data-generating process is indeed linear. However, when $X_1$ influences the outcome $Y$ not only through a linear, but also through a cubic term, that is (M4), we discern that PSI based on a linear model has no power at all whereas model-agnostic methods still achieve noticeable power. This exemplifies that PSI procedures, built upon a certain model, can only be confidently used if there is limited uncertainty about the underlying data generation.
\begin{figure*}[!ht]
    \vskip 0.1in
    \begin{center}
    \centerline{\includegraphics[width=15cm, height=3.7cm]{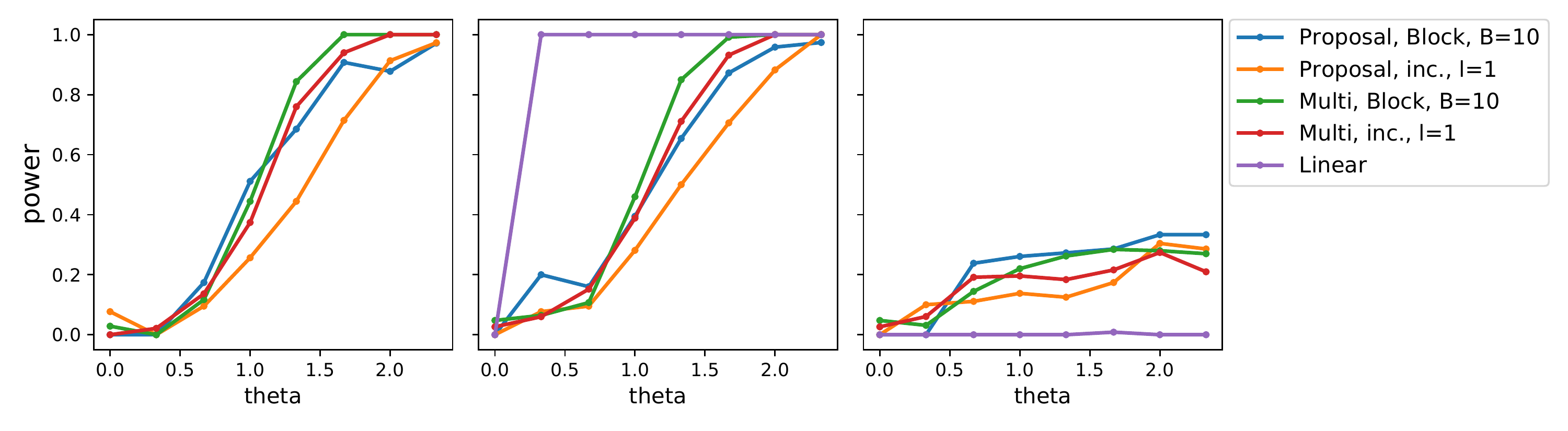}}
    \caption{Power of detecting $X_1$ as influential feature in different toy models (from left to right): (M1'), (M3), (M4)}
    \label{fig:power}
    \end{center}
    \vskip -0.2in
\end{figure*}




\subsubsection{Additional Experiments}
In our simulations, we focused on the statistical properties of the proposed approach but also started investigating the behaviour of the algorithm for different feature selection set-ups. The conclusions we drew deserve a few comments.\newline
\textit{\textbf{Screening}} HSIC-ordering includes the influential features into the screened set with high probability when $p'$ and the size of the first fold are sufficiently large. Nonetheless, screening is merely a method to reduce computational complexity and can potentially weaken the performance of the downstream HSIC-Lasso procedure. Therefore, the number of screened features $p'$ should be set as high as computational resources allow.\newline
\textit{\textbf{Regularisation}} The use of an adaptive penalty term generally leads to fewer selected features than the vanilla Lasso regularisation. Moreover, cross-validation and the AIC often choose similar values for $\lambda$.\newline
\textit{\textbf{Dependence structure}} In datasets with strong correlation between influential and uninfluential features we observe that the partial target is capable of correctly detecting the uninfluential ones. How the dependence structure among influential features materialises in rejections of the null hypothesis, however, is a subtle and still open question.\newline
\textit{\textbf{Estimators}} In general, block estimators are computationally less expensive than incomplete U-statistics estimators. For the latter ones, the calculation costs, but also the power increases with the size $l$.\newline
\textit{\textbf{Split ratio}} In the experiments with artificial data, the size of the first fold was set to be a third of the second fold as this already suffices to obtain a decent estimate of the regularisation parameter. For most datasets it is advisable to dedicate more data to the HSIC-Lasso procedure than to the hyper-parameter selection; however, a reliable heuristic for the split ratio remains subject to future research.

\subsection{Benchmarks}
Now we proceed to applying our proposed algorithm to benchmark datasets from the UCI Repository and the Broad Institute’s Single Cell Portal, respectively. We provide an additional, more in-depth experiment in the supplementary materials.
\subsubsection{Turkish Student Dataset}
This dataset contains 5\,820 course evaluation scores provided by students from Gazi University, Ankara, who answered 28 questions on a five-level Likert scale, see further \cite{turkish-student}. For our experiment we use the perceived difficulty of the course as response variable.\newline
This data was previously evaluated by \citet{yamada2018} who selected features with HSIC-ordering and used the familiar framework of the polyhedral lemma and truncated Gaussians for PSI, denoted by Poly. We use the block estimator of size 10 and set $k=10$ for the Multi approach to accord with \citeauthor{yamada2018} and report their findings along with ours. For our proposal, the first fold contains 20\% of the data, we select \(\lambda\) with 10-fold cross-validation and do not carry out screening as the number of features (\(p=28\)) is manageable. Table \ref{tab:student} summarises our findings.
\begin{table}[t]
        \caption{P-values of the HSIC-target for selected features in the Turkish student dataset}
        \label{tab:student}
        \vskip 0.1in
        \begin{center}
    \begin{small}
    \begin{sc}
        \begin{tabular}{|c|c|c|c|}
                \hline
                \multicolumn{1}{|c|}{Feature} & \multicolumn{3}{c|}{p-value} \\ \cline{2-4}
                & Proposal & Multi & Poly\\\hline
                Q2 & \textbf{0.021} & - & 0.452\\
                Q3 & - & 0.782 & -\\
                Q11 & \textbf{0.004} & - &- \\
                Q13 & - & - & \textbf{0.018}\\
                Q14 & - & \textbf{0.001} & -\\
                Q15 & - & 0.095 & -\\
                Q17 & $<$\textbf{0.001} & $<$\textbf{0.001} & \textbf{0.033}\\
                Q18 & - & - & 0.186\\
                Q19 & - & $<$\textbf{0.001} & -\\
                Q20 & - & \textbf{0.004} & 0.463\\
                Q21 & - & \textbf{0.032} & \textbf{0.033}\\
                Q22 & - & $<$\textbf{0.001} & \textbf{0.042}\\
                Q23 & - & - & \textbf{0.037}\\
                Q25 & - & \textbf{0.002} & - \\
                Q26 & - & - & 0.176\\
                Q28 & \textbf{0.004} & \textbf{0.041} & $<$\textbf{0.001}\\
                \hline
        \end{tabular}
    \end{sc}
    \end{small}
        \end{center}
        \vskip -0.1in
\end{table}
\newline
First, we notice that Multi and Poly pick different features despite sharing the same selection procedure which can probably be attributed to randomisation, carried out by \citet{yamada2018}. Moreover, we observe that HSIC-Lasso chooses a very parsimonious model with only four covariates whose associated HSIC-targets are highly significant. Among the tested approaches, there is a moderate agreement on the influential covariates where only 'Q17: The Instructor arrived on time for classes.' and 'Q28: The Instructor treated all students in a right and objective manner.' are unanimously chosen and found to be significant. Considering the partial targets for HSIC-Lasso, we find that only \(\bpar_{17,S}\) appears significant.\newline
The different results that we obtain may hint that the Turkish student data set is intrinsically noisy or that methods based on HSIC-estimation and the polyhedral lemma are unstable. However, the Lasso-selection of HSIC-Lasso, unlike Multi or Poly, penalises correlated features which affects PSI as well. Hence, different results for the compared methods do not necessarily indicate incorrectness of either approach.

\subsubsection{Single-Cell RNAseq Data}
\citet{villani} isolated around 2\,400 blood cells, enriched in two particular kinds of leukocytes: dendritic cells (DCs) and monocytes. Then, they measured the gene expression on every cell using single-cell RNAseq aiming to describe the diversity between, and within those two cell types based on their gene expression profile. They end up defining 10 different subclasses: 6 types of DCs and 4 types of monocytes. In our experiment, we use 1\,078 samples of this data aspiring to find the genes that separate these 10 classes among the 26\,593 genes. We standardise the single-cell RNAseq data gene-wise and impute missing gene expressions with MAGIC, see further \cite{magic}. Since the response is categorical, we use the normalised delta kernel. Unlike the Turkish student dataset, we are now confronted with a considerably high-dimensional problem where the number of features greatly exceeds the sample size. Therefore, screening becomes more challenging and we consequently split the data evenly into first and second fold. We screen 1\,000 potentially influential features and apply the incomplete U-statistic estimator with a large size of 20, hoping to better capture the potentially involved dependence structure. The remaining parameters were set as in the previous experiment. For an in-depth analysis of the dataset, we recommend to conduct a sensitivity analysis which investigates the behaviour of the method for different values of the parameters, such as the split ratio, the number of screened features or the size of the estimator.\newline
We find that HSIC-Lasso selects 13 features and that all of the associated HSIC-targets and most of the partial targets are significant, cf. Table~\ref{tab:villani}.
\begin{table}[t]
        \caption{P-Values of the partial targets corresponding to selected features in the single-cell RNAseq dataset}
        \label{tab:villani}
        \vskip 0.1in
        \begin{center}
    \begin{small}
    \begin{sc}
    \begin{tabular}{|c|c||c|c|}
                \hline
                Gene & p-value & Gene & p-value\\\hline
                \emph{ACTB} & 0.961 & \emph{IGJ} & $<$\textbf{0.001} \\
                \emph{CD14} & \textbf{0.026} & \emph{LYZ} & $<$\textbf{0.001} \\
                \emph{FCER1A} & $<$\textbf{0.001} & \emph{MTRNR2L2} & 0.420 \\
                \emph{FCGR3A} & \textbf{0.001} & \emph{RPS3A} & $<$\textbf{0.001} \\
                \emph{FTL} & 0.968 & \emph{TMSB4X} & \textbf{0.012} \\
                \emph{HLA-DPA1} & $<$\textbf{0.001} & \emph{TVAS5} & 0.553 \\
                \emph{IFI30} & \textbf{0.002} &  & \\\hline
        \end{tabular}
        \end{sc}
    \end{small}
        \end{center}
        \vskip -0.1in
\end{table}
One of the traditionally defining characteristics of monocytes is the expression of the CD14 protein; encouragingly, HSIC-Lasso selected this gene as a discriminating feature. In fact, it also selected six other genes which \citeauthor{villani} used in multiple cell signatures: \emph{FCGR3A} (DC4), \emph{FCER1A} (DC2 and DC3), \emph{FTL} (DC4), \emph{IFI30} (cDC-like), \emph{IGJ} (pDC-like), and \emph{LYZ} (cDC-like). Jointly, this shows the ability of HSIC-Lasso to recover multiple genes used to define the classes.\newline
More exciting, however, are the genes which are selected by HSIC-Lasso and were not used in the original study. These might point to new molecular signatures and functions that differentiate these cell types. For instance, one of these genes play a role in immunity: \emph{HLA-DPA1}, which presenting cells like DCs use to present exogenous proteins to other immune cells. The proposed PSI framework adds nuance to this picture by providing a soft ranking of the selected genes. Notoriously, \emph{FTL} and \emph{ACTB} have p-values close to 1, which suggests that most of the information they provide is already captured by other selected genes. Hence, other genes should be prioritised in hypothetical downstream experiments.\newline
Our method contributes to the detection of discriminatory features inasmuch it facilitates to formulate confidence statements about the obtained results.

\section{Discussion and Outlook}
Post-selection inference was initially proposed for linear regression models with Lasso regularisation and subsequently expanded to generalised regression situations and Lasso variants, see for example \cite{l1_pen} and \cite{gen1}. However, knowledge of an underlying model seemed to be necessary in order to properly account for the selection-process. \citet{yamada2018} overcame this limitation by capturing the unknown dependence between variables via the model-agnostic Hilbert-Schmidt independence criterion and embedding the estimates into the formerly proposed PSI-framework using asymptotic normality. However, this approach still requires the user to decide how many features to select. Our method chooses features in a data-driven manner and correctly accounts for the selection process and is thus ideal for situations with limited knowledge about the structure of the data.\newline
Extending the theoretical framework to allow for a sequential application of HSIC-Lasso with different values of $\lambda$, similar to \citeauthor{taylor14}'s least angle regression algorithm (\citeyear{taylor14}), is an interesting and practically relevant step for future research. Moreover, \citet{crude_targets} hint that developing inference targets apart from the partial and HSIC-target can be useful to reduce the length of confidence intervals. Lastly, the incorporation of the choice of $\lambda$ into the post-selection framework for HSIC-Lasso would allow to analyse the data on only one fold and render the proposed framework fully in line with the PSI philosophy. \citet{loftus} and \citet{psi_cv} took first steps in this direction; albeit, the application to HSIC-Lasso is still an open issue.\newline
From an algorithmic point of view, establishing heuristics for a good choice of hyper-parameters, such as the size of the estimators or the split ratio, can as well be object of future research.

\section*{Acknowledgments}
Tobias Freidling thanks the Max Weber Program and Erasmus+ for supporting his stay at Kyoto University with a scholarship and Mathias Drton for an insightful discussion on Theorem \ref{th:asymp_pivot}. Makoto Yamada was supported by MEXT KAKENHI 20H04243 and partly supported by MEXT KAKENHI 21H04874. Benjamin Poignard was supported by the Japanese Society for the Promotion of Science (Start-Support Kakenhi 19K23193).

\bibliography{references}

\begin{thebibliography}{50}
\providecommand{\natexlab}[1]{#1}
\providecommand{\url}[1]{\texttt{#1}}
\expandafter\ifx\csname urlstyle\endcsname\relax
  \providecommand{\doi}[1]{doi: #1}\else
  \providecommand{\doi}{doi: \begingroup \urlstyle{rm}\Url}\fi

\bibitem[{Akaike}(1974)]{AIC}
{Akaike}, H.
\newblock {A new look at the statistical model identification.}
\newblock \emph{{IEEE Transactions on Automatic Control}}, 19:\penalty0
  716--723, 1974.

\bibitem[{Aronszajn}(1950)]{aronszajn}
{Aronszajn}, N.
\newblock {Theory of reproducing kernels.}
\newblock \emph{{Transactions of the American Mathematical Society}},
  68:\penalty0 337--404, 1950.

\bibitem[{Berk} et~al.(2013){Berk}, {Brown}, {Buja}, {Zhang}, and {Zhao}]{posi}
{Berk}, R., {Brown}, L., {Buja}, A., {Zhang}, K., and {Zhao}, L.
\newblock {Valid post-selection inference.}
\newblock \emph{{The Annals of Statistics}}, 41\penalty0 (2):\penalty0
  802--837, 2013.

\bibitem[{Borgwardt} et~al.(2006){Borgwardt}, {Gretton}, {Rasch}, {Kriegel},
  {Sch\"olkopf}, and {Smola}]{mmd_borgwardt}
{Borgwardt}, K.~M., {Gretton}, A., {Rasch}, M.~J., {Kriegel}, H.-P.,
  {Sch\"olkopf}, B., and {Smola}, A.~J.
\newblock {Integrating structured biological data by Kernel Maximum Mean
  Discrepancy}.
\newblock \emph{Bioinformatics}, 22\penalty0 (14):\penalty0 49--57, 2006.

\bibitem[{Boyd} \& {Vandenberghe}(2004){Boyd} and {Vandenberghe}]{boyd}
{Boyd}, S. and {Vandenberghe}, L.
\newblock \emph{{Convex optimization.}}
\newblock Cambridge University Press, 2004.

\bibitem[{Chen} et~al.(2010){Chen}, {Wiesel}, {Eldar}, and {Hero}]{oas}
{Chen}, Y., {Wiesel}, A., {Eldar}, Y.~C., and {Hero}, A.~O.
\newblock {Shrinkage algorithms for MMSE covariance estimation.}
\newblock \emph{{IEEE Transactions on Signal Processing}}, 58\penalty0
  (10):\penalty0 5016--5029, 2010.

\bibitem[{Cox}(1975)]{sample_splitting}
{Cox}, D.~R.
\newblock {A note on data-splitting for the evaluation of significance levels.}
\newblock \emph{{Biometrika}}, 62:\penalty0 441--444, 1975.

\bibitem[Dua \& Graff(2017)Dua and Graff]{uci}
Dua, D. and Graff, C.
\newblock {UCI} machine learning repository, 2017.

\bibitem[{Efron} et~al.(2004){Efron}, {Hastie}, {Johnstone}, and
  {Tibshirani}]{lar}
{Efron}, B., {Hastie}, T., {Johnstone}, I., and {Tibshirani}, R.
\newblock {Least angle regression. (With discussion).}
\newblock \emph{{The Annals of Statistics}}, 32\penalty0 (2):\penalty0
  407--499, 2004.

\bibitem[{Friedman} et~al.(2007){Friedman}, {Hastie}, {H\"ofling}, and
  {Tibshirani}]{coordinate_descent}
{Friedman}, J., {Hastie}, T., {H\"ofling}, H., and {Tibshirani}, R.
\newblock {Pathwise coordinate optimization}.
\newblock \emph{{The Annals of Applied Statistics}}, 1\penalty0 (2):\penalty0
  302--332, 2007.

\bibitem[{Gretton} et~al.(2005{\natexlab{a}}){Gretton}, {Bousquet}, {Smola},
  and {Sch\"olkopf}]{gretton}
{Gretton}, A., {Bousquet}, O., {Smola}, A., and {Sch\"olkopf}, B.
\newblock {Measuring statistical dependence with Hilbert-Schmidt norms.}
\newblock In \emph{{Algorithmic learning theory. 16th international conference,
  ALT 2005, Singapore, October 8--11, 2005. Proceedings.}}, pp.\  63--77.
  Berlin: Springer, 2005{\natexlab{a}}.

\bibitem[{Gretton} et~al.(2005{\natexlab{b}}){Gretton}, {Smola}, {Bousquet},
  {Herbrich}, {Belitski}, {Augath}, {Murayama}, {Pauls}, {Sch\"olkopf}, and
  {Logothetis}]{coco}
{Gretton}, A., {Smola}, A., {Bousquet}, O., {Herbrich}, R., {Belitski}, A.,
  {Augath}, M., {Murayama}, Y., {Pauls}, J., {Sch\"olkopf}, B., and
  {Logothetis}, N.
\newblock {Kernel Constrained Covariance for Dependence Measurement}.
\newblock In \emph{Proceedings of the Tenth International Workshop on
  Artificial Intelligence and Statistics}, pp.\  1--8, 2005{\natexlab{b}}.

\bibitem[{Gretton} et~al.(2012){Gretton}, {Borgwardt}, {Rasch}, {Sch\"olkopf},
  and {Smola}]{mmd}
{Gretton}, A., {Borgwardt}, K.~M., {Rasch}, M.~J., {Sch\"olkopf}, B., and
  {Smola}, A.
\newblock {A kernel two-sample test.}
\newblock \emph{{Journal of Machine Learning Research (JMLR)}}, 13:\penalty0
  723--773, 2012.

\bibitem[{Gunduz} \& {Fokoue}(2013){Gunduz} and {Fokoue}]{turkish-student}
{Gunduz}, N. and {Fokoue}, E.
\newblock {UCI} machine learning repository, 2013.

\bibitem[{Higham}(1988)]{higham}
{Higham}, N.~J.
\newblock {Computing a nearest symmetric positive semidefinite matrix}.
\newblock \emph{Linear Algebra and its Applications}, 103:\penalty0 103 -- 118,
  1988.

\bibitem[{Hocking}(1976)]{forward-stepwise}
{Hocking}, R.~R.
\newblock {A Biometrics Invited Paper. The Analysis and Selection of Variables
  in Linear Regression}.
\newblock \emph{Biometrics}, 32\penalty0 (1):\penalty0 1--49, 1976.

\bibitem[{Hoeffding}(1948)]{hoeffding}
{Hoeffding}, W.
\newblock {A class of statistics with asymptotically normal distribution.}
\newblock \emph{{Annals of Mathematical Statistics}}, 19:\penalty0 293--325,
  1948.

\bibitem[{Hyun} et~al.(2018){Hyun}, {G'sell}, and {Tibshirani}]{gen1}
{Hyun}, S., {G'sell}, M., and {Tibshirani}, R.~J.
\newblock {Exact post-selection inference for the generalized lasso path.}
\newblock \emph{{Electronic Journal of Statistics}}, 12\penalty0 (1):\penalty0
  1053--1097, 2018.

\bibitem[{Korolyuk} \& {Borovskikh}(1994){Korolyuk} and
  {Borovskikh}]{u_stats_rus}
{Korolyuk}, V.~S. and {Borovskikh}, Y.~V.
\newblock \emph{{Theory of \(U\)-statistics. Updated and transl. from the
  Russian by P. V. Malyshev and D. V. Malyshev.}}
\newblock Dordrecht: Kluwer Academic Publishers, 1994.

\bibitem[{Lee}(1990)]{lee_u_stats}
{Lee}, A.~J.
\newblock \emph{{\(U\)-statistics. Theory and practice.}}, volume 110.
\newblock New York etc.: Marcel Dekker, Inc., 1990.

\bibitem[{Lee} et~al.(2016){Lee}, {Sun}, {Sun}, and {Taylor}]{lee}
{Lee}, J.~D., {Sun}, D.~L., {Sun}, Y., and {Taylor}, J.~E.
\newblock {Exact post-selection inference, with application to the Lasso.}
\newblock \emph{{The Annals of Statistics}}, 44\penalty0 (3):\penalty0
  907--927, 2016.

\bibitem[{Leeb} \& {P\"otscher}(2005){Leeb} and {P\"otscher}]{leeb-2005}
{Leeb}, H. and {P\"otscher}, B.~M.
\newblock {Model selection and inference: facts and fiction.}
\newblock \emph{{Econometric Theory}}, 21\penalty0 (1):\penalty0 21--59, 2005.

\bibitem[{Leeb} \& {P\"otscher}(2006){Leeb} and {P\"otscher}]{leeb-2006}
{Leeb}, H. and {P\"otscher}, B.~M.
\newblock {Can one estimate the conditional distribution of
  post-model-selection estimators?}
\newblock \emph{{The Annals of Statistics}}, 34\penalty0 (5):\penalty0
  2554--2591, 2006.

\bibitem[{Lehmann} \& {Romano}(2005){Lehmann} and {Romano}]{bible}
{Lehmann}, E.~L. and {Romano}, J.~P.
\newblock \emph{{Testing statistical hypotheses. 3rd ed.}}
\newblock New York, NY: Springer, 3rd ed. edition, 2005.

\bibitem[{Lim} et~al.(2020){Lim}, {Yamada}, {Jitkrittum}, {Terada}, {Matsui},
  and {Shimodaira}]{aistats}
{Lim}, J.~N., {Yamada}, M., {Jitkrittum}, W., {Terada}, Y., {Matsui}, S., and
  {Shimodaira}, H.
\newblock {More Powerful Selective Kernel Tests for Feature Selection}.
\newblock volume 108 of \emph{Proceedings of Machine Learning Research}, pp.\
  820--830. PMLR, 2020.

\bibitem[{Liu} et~al.(2018){Liu}, {Markovic}, and {Tibshirani}]{crude_targets}
{Liu}, K., {Markovic}, J., and {Tibshirani}, R.
\newblock {More powerful post-selection inference, with application to the
  Lasso}.
\newblock \emph{arXiv preprint}, 1801.09037, 2018.

\bibitem[{Loftus}(2015)]{loftus}
{Loftus}, J.~R.
\newblock {Selective inference after cross-validation}.
\newblock \emph{arXiv preprint}, 1511.08866, 2015.

\bibitem[{Markovic} et~al.(2017){Markovic}, {Xia}, and {Taylor}]{psi_cv}
{Markovic}, J., {Xia}, L., and {Taylor}, J.
\newblock {Unifying approach to selective inference with applications to
  cross-validation}.
\newblock \emph{arXiv preprint}, 1703.06559, 2017.

\bibitem[Redmond \& Baveja(2002)Redmond and Baveja]{redmond}
Redmond, M. and Baveja, A.
\newblock A data-driven software tool for enabling cooperative information
  sharing among police departments.
\newblock \emph{European Journal of Operational Research}, 141\penalty0
  (3):\penalty0 660--678, 2002.

\bibitem[{Sch\"olkopf} \& {Smola}(2018){Sch\"olkopf} and
  {Smola}]{median-heuristic}
{Sch\"olkopf}, B. and {Smola}, A.~J.
\newblock \emph{{Learning with Kernels: Support Vector Machines,
  Regularization, Optimization, and Beyond}}.
\newblock Cambridge, MA: MIT Press, 2018.

\bibitem[{Shimodaira}(2004)]{shim2004}
{Shimodaira}, H.
\newblock {Approximately unbiased tests of regions using multistep-multiscale
  bootstrap resampling}.
\newblock \emph{{The Annals of Statistics}}, 32\penalty0 (6):\penalty0
  2616--2641, 2004.

\bibitem[{Smola} et~al.(2007){Smola}, {Gretton}, {Song}, and
  {Sch\"olkopf}]{HS_embed}
{Smola}, A., {Gretton}, A., {Song}, L., and {Sch\"olkopf}, B.
\newblock {A Hilbert space embedding for distributions.}
\newblock In \emph{{Algorithmic learning theory. 18th international conference,
  ALT 2007, Sendai, Japan, October 1--4, 2007. Proceedings}}, pp.\  13--31.
  Berlin: Springer, 2007.

\bibitem[{Song} et~al.(2012){Song}, {Smola}, {Gretton}, {Bedo}, and
  {Borgwardt}]{song}
{Song}, L., {Smola}, A., {Gretton}, A., {Bedo}, J., and {Borgwardt}, K.
\newblock {Feature selection via dependence maximization.}
\newblock \emph{{Journal of Machine Learning Research (JMLR)}}, 13:\penalty0
  1393--1434, 2012.

\bibitem[{Steinwart}(2002)]{steinwart}
{Steinwart}, I.
\newblock {On the influence of the kernel on the consistency of support vector
  machines.}
\newblock \emph{{Journal of Machine Learning Research (JMLR)}}, 2\penalty0
  (1):\penalty0 67--93, 2002.

\bibitem[{Stone}(1974)]{cross_val}
{Stone}, M.
\newblock {Cross-validatory choice and assessment of statistical predictions.
  Discussion}.
\newblock \emph{{Journal of the Royal Statistical Society. Series B}},
  36:\penalty0 111--147, 1974.

\bibitem[{Taylor} \& {Tibshirani}(2018){Taylor} and {Tibshirani}]{l1_pen}
{Taylor}, J. and {Tibshirani}, R.
\newblock {Post-selection inference for \(\ell_1\)-penalized likelihood
  models.}
\newblock \emph{{The Canadian Journal of Statistics}}, 46\penalty0
  (1):\penalty0 41--61, 2018.

\bibitem[{Tibshirani}(1996)]{lasso}
{Tibshirani}, R.
\newblock {Regression shrinkage and selection via the lasso.}
\newblock \emph{{Journal of the Royal Statistical Society. Series B}},
  58\penalty0 (1):\penalty0 267--288, 1996.

\bibitem[{Tibshirani} et~al.(2016){Tibshirani}, {Taylor}, {Lockhart}, and
  {Tibshirani}]{taylor14}
{Tibshirani}, R.~J., {Taylor}, J., {Lockhart}, R., and {Tibshirani}, R.
\newblock {Exact Post-Selection Inference for Sequential Regression
  Procedures}.
\newblock \emph{{Journal of the American Statistical Association}},
  111\penalty0 (514):\penalty0 600--620, 2016.

\bibitem[{Tibshirani} et~al.(2018){Tibshirani}, {Rinaldo}, {Tibshirani}, and
  {Wasserman}]{high_dim_gen1}
{Tibshirani}, R.~J., {Rinaldo}, A., {Tibshirani}, R., and {Wasserman}, L.
\newblock {Uniform asymptotic inference and the bootstrap after model
  selection.}
\newblock \emph{{The Annals of Statistics}}, 46\penalty0 (3):\penalty0
  1255--1287, 2018.

\bibitem[{U. S. Department of Commerce, Bureau of the Census}()]{data1}
{U. S. Department of Commerce, Bureau of the Census}.
\newblock Census of population and housing 1990 united states: Summary tape
  file 1a \& 3a (computer files).

\bibitem[{U.S. Department Of Commerce, Bureau Of The Census Producer,
  Washington, DC} \& {Inter-university Consortium for Political and Social
  Research Ann Arbor, Michigan}(1992){U.S. Department Of Commerce, Bureau Of
  The Census Producer, Washington, DC} and {Inter-university Consortium for
  Political and Social Research Ann Arbor, Michigan}]{data2}
{U.S. Department Of Commerce, Bureau Of The Census Producer, Washington, DC}
  and {Inter-university Consortium for Political and Social Research Ann Arbor,
  Michigan}, 1992.

\bibitem[{U.S. Department of Justice, Bureau of Justice
  Statistics}(1992)]{data3}
{U.S. Department of Justice, Bureau of Justice Statistics}.
\newblock Law enforcement management and administrative statistics (computer
  file), 1992.

\bibitem[{U.S. Department of Justice, Federal Bureau of
  Investigation}(1995)]{data4}
{U.S. Department of Justice, Federal Bureau of Investigation}.
\newblock Crime in the united states (computer file), 1995.

\bibitem[{van der Vaart}(1998)]{van_der_vaart}
{van der Vaart}, A.~W.
\newblock \emph{{Asymptotic statistics}}, volume~3.
\newblock Cambridge: Cambridge Univ. Press, 1998.

\bibitem[{van Dijk} et~al.(2018){van Dijk}, {Sharma}, {Nainys}, {Yim},
  {Kathail}, {Carr}, {Burdziak}, {Moon}, {Chaffer}, {Pattabiraman}, {Bierie},
  {Mazutis}, {Wolf}, {Krishnaswamy}, and {Pe’er}]{magic}
{van Dijk}, D., {Sharma}, R., {Nainys}, J., {Yim}, K., {Kathail}, P., {Carr},
  A.~J., {Burdziak}, C., {Moon}, K.~R., {Chaffer}, C.~L., {Pattabiraman}, D.,
  {Bierie}, B., {Mazutis}, L., {Wolf}, G., {Krishnaswamy}, S., and {Pe’er},
  D.
\newblock {Recovering Gene Interactions from Single-Cell Data Using Data
  Diffusion}.
\newblock \emph{Cell}, 174\penalty0 (3):\penalty0 716 -- 729.e27, 2018.

\bibitem[{Villani} et~al.(2017){Villani}, {Satija}, {Reynolds}, {Sarkizova},
  {Shekhar}, {Fletcher}, {Griesbeck}, {Butler}, {Zheng}, {Lazo}, {Jardine},
  {Dixon}, {Stephenson}, {Nilsson}, {Grundberg}, {McDonald}, {Filby}, {Li}, {De
  Jager}, {Rozenblatt-Rosen}, {Lane}, {Haniffa}, {Regev}, and
  {Hacohen}]{villani}
{Villani}, A.-C., {Satija}, R., {Reynolds}, G., {Sarkizova}, S., {Shekhar}, K.,
  {Fletcher}, J., {Griesbeck}, M., {Butler}, A., {Zheng}, S., {Lazo}, S.,
  {Jardine}, L., {Dixon}, D., {Stephenson}, E., {Nilsson}, E., {Grundberg}, I.,
  {McDonald}, D., {Filby}, A., {Li}, W., {De Jager}, P.~L., {Rozenblatt-Rosen},
  O., {Lane}, A.~A., {Haniffa}, M., {Regev}, A., and {Hacohen}, N.
\newblock Single-cell rna-seq reveals new types of human blood dendritic cells,
  monocytes, and progenitors.
\newblock \emph{Science}, 356\penalty0 (6335), 2017.

\bibitem[{Yamada} et~al.(2014){Yamada}, {Jitkrittum}, {Sigal}, {Xing}, and
  {Sugiyama}]{hsic_lasso}
{Yamada}, M., {Jitkrittum}, W., {Sigal}, L., {Xing}, E.~P., and {Sugiyama}, M.
\newblock {High-dimensional feature selection by feature-wise kernelized
  Lasso.}
\newblock \emph{{Neural Computation}}, 26\penalty0 (1):\penalty0 185--207,
  2014.

\bibitem[{Yamada} et~al.(2018){Yamada}, {Umezu}, {Fukumizu}, and
  {Takeuchi}]{yamada2018}
{Yamada}, M., {Umezu}, Y., {Fukumizu}, K., and {Takeuchi}, I.
\newblock {Post Selection Inference with Kernels}.
\newblock volume~84 of \emph{Proceedings of Machine Learning Research}, pp.\
  152--160. PMLR, 2018.

\bibitem[{Zhang} et~al.(2018){Zhang}, {Filippi}, {Gretton}, and
  {Sejdinovic}]{zhang}
{Zhang}, Q., {Filippi}, S., {Gretton}, A., and {Sejdinovic}, D.
\newblock {Large-scale kernel methods for independence testing.}
\newblock \emph{{Statistics and Computing}}, 28\penalty0 (1):\penalty0
  113--130, 2018.

\bibitem[{Zou}(2006)]{adaptive_lasso}
{Zou}, H.
\newblock {The adaptive lasso and its oracle properties.}
\newblock \emph{{Journal of the American Statistical Association}},
  101\penalty0 (476):\penalty0 1418--1429, 2006.

\end{thebibliography}
\bibliographystyle{icml2021}
\newpage
\onecolumn
\appendix

\section{Technical Appendix on HSIC and HSIC-Lasso}

\subsection{Measuring Dependence with HSIC}


        The main incentive to develop advanced techniques to describe dependence relations between two random variables \(X\) and \(Y\) arises from the fact that the covariance
        \begin{equation*}
        \cov{X}{Y} = \E{XY}-\E{X}\E{Y},
        \end{equation*}
        is designed for linear relationships only. If the dependence structure, however, is of non-linear nature, the covariance can only partly capture the relationship between \(X\) and \(Y\) or completely fails to do so. Nevertheless, general, or rather model-free, independence can be expressed in terms of the covariance as follows, cf. \cite{coco}.
        \begin{proposition}\label{prop:gen_indep}
                The random variables \(X\) and \(Y\) are independent if and only if  \(\cov{f(X)}{g(Y)} = 0\) for each pair \((f,g)\) of bounded, continuous functions.
        \end{proposition}
        There are two lines of thought leading to the Hilbert-Schmidt independence criterion: one presented by \citet{gretton} regarding HSIC as the Hilbert-Schmidt norm of a cross-covariance operator and one thinking of HSIC as maximum mean discrepancy on a product space according to \citet{zhang}. In this work, we follow the latter derivation and link it with the first approach and Proposition \ref{prop:gen_indep} at the end.\newline
        First, we introduce the concept of reproducing kernel Hilbert spaces.
        \begin{definition}
                Let \(\Hc\) be a Hilbert space of real-valued functions defined on \(D\) with scalar product \(\langle\cdot,\cdot\rangle_\Hc\). A function \(k\colon D\times D \rightarrow \R \) is called a \emph{reproducing kernel} of \(\Hc\) if
                \begin{flalign*}
                &\text{1. } k(\cdot,x) \in \Hc\quad \forall x\in D,&&\\
                &\text{2. } \langle f, k(\cdot, x) \rangle_\Hc = f(x) \quad\forall x\in D\,\,\forall f \in \Hc.&&
                \end{flalign*}
                If \(\Hc\) has a reproducing kernel, it is called a \emph{reproducing kernel Hilbert space (RKHS)}.
        \end{definition}
        \begin{remark}
                As an immediate consequence of the upper definition, we get
                \begin{equation*}
                k(x,y) = \langle k(\cdot, x), k(\cdot, y)\rangle_\Hc\quad \forall x,y\in D.
                \end{equation*}
        \end{remark}
        The following theorem, proved by \citet{aronszajn}, provides sufficient conditions for a function \(k\) to be a reproducing kernel.
        \begin{theorem}[Moore-Aronszajn]
                Let \(k\colon D\times D \rightarrow \R\), be symmetric and positive definite, that is
                \begin{equation*}
                \sum_{i = 1}^n \sum_{j = 1}^n a_i a_j k(x_i, x_j) \geq 0, \quad \forall n\geq 1\,\,\forall a \in \R^n\,\,\forall x \in D^n.
                \end{equation*}
                Then there is a unique RKHS \(\Hc_k\) with reproducing kernel \(k\).
        \end{theorem}
        Against this backdrop, we may ask how properties of the kernel \(k\) translate into characteristics of \(\Hc_k\). The notion of a universal kernel, introduced by \citet{steinwart}, helps to shed light on this issue.
        \begin{definition}\label{def:univ_kernel}
                A continuous kernel \(k\) on a compact metric space \((D, d)\) is called \emph{universal} if \(\Hc_k\) is dense in \(C(D)\), the space of continuous functions on \(D\), with respect to \(\lVert \cdot\rVert_\infty\).
        \end{definition}
        It is shown that both the Gaussian and exponential kernel, defined by
        \begin{equation*}
        k(x, y) = \exp\left(-\frac{\lVert x-y \rVert_2^2}{2\sigma^2}\right),\,\sigma^2 > 0,\qquad
        k(x, y) = \exp\left(-\frac{\lVert x-y \rVert_2}{2\sigma}\right),\, \sigma > 0,
        \end{equation*}
        respectively, are universal.\newline
        Second, we introduce the particularly useful framework of embedding distributions into Hilbert spaces according to \citet{HS_embed}.
        \begin{definition}\label{prob_embed}
                Let \(k\) be a bounded kernel on \(D\) and \(\Prob\) a probability measure on \(D\). The \emph{kernel embedding} of \(\Prob\) into the RKHS \(\Hc_k\) is \(\mu_k(\Prob) \in \Hc_k\) such that
                \begin{equation*}
                \E{f(X)} = \int_D f(x)\,\text{d}\Prob(x) =  \langle f, \mu_k(\Prob)\rangle_{\Hc_k}, \quad X\sim \Prob,\,\forall f \in \Hc_k.
                \end{equation*}
        \end{definition}
        \begin{remark}\label{alt_embed}
                Alternatively, \(\mu_k(\Prob)\) can be defined by
                \begin{equation*}
                \mu_k(\Prob) = \int_D k(\cdot, x)\,\text{d}\Prob(x).
                \end{equation*}
        \end{remark}
        Definition \ref{prob_embed} allows us to use Hilbert space theory on distributions which gives rise to the definition of maximum mean discrepancy (MMD), see for example \cite{mmd_borgwardt} and \cite{mmd}, which measures the distance between probability measures.
        \begin{definition}\label{def:MMD}
                Let \(k\) be a bounded kernel and \(\Prob\) and \(\Qrob\) probability measures on \(D\). The \emph{maximum mean discrepancy (MMD)} between \(\Prob\) and \(\Qrob\) with respect to \(k\) is defined as
                \begin{equation*}
                \MMD_k(\Prob, \Qrob) = \lVert \mu_k(\Prob) - \mu_k(\Qrob)\rVert^2_{\Hc_k}.
                \end{equation*}
        \end{definition}

        \begin{lemma}\label{lem:univ_kernel}
                In the setting of Definition \ref{def:MMD}, \(\MMD_k\) is a metric on probability measures if \(k\) is a universal kernel.
        \end{lemma}
        \begin{proof}
                Theorem 1 of \cite{HS_embed} states that \(\Prob \mapsto \mu_k(\Prob)\) is injective for universal \(k\). Hence, any two different measures have two distinct embeddings. The statement directly follows from the norm properties of \(\lVert\cdot\rVert_{\Hc_k}\).
        \end{proof}
        The maximum mean discrepancy can be used to test whether two given data samples stem from the same distribution. Since our goal is to find a measure for the dependence between two random variables \(X\) and \(Y\), we use MMD to compare the joint distribution \(\Prob_{X,Y}\) and the product of the marginals \(\Prob_X \Prob_Y\).\newline
        To this end, consider any two kernels \(k\) and \(l\) on the domains \(D_X\) and \(D_Y\). It is easy to verify that \(K = k \otimes l\), given by
        \begin{equation*}
        K\big((x, y), (x', y')\big) = k(x, x')\,l(y, y'),\quad
        x,x'\in D_X,\,\, y,y' \in D_Y,
        \end{equation*}
        is a valid kernel on the product space \(D_X\times D_Y\). Employing Remark \ref{alt_embed}, we can define a dependence measure between \(X\) and \(Y\) based on RKHSs.
        \begin{definition}\label{def:hsic}
                Let \(X\) and \(Y\) be random variables and \(k\) and \(l\) be bounded kernels on the domains \(D_X\) and \(D_Y\), respectively. The \emph{Hilbert-Schmidt independence criterion} \(\HSIC_{k,l}(X,Y)\) for \(X\) and \(Y\) based on the kernels \(k\) and \(l\) is given by
                \begin{align}
                \HSIC_{k,l}(X,Y) &= \MMD_{k\otimes l}(\Prob_{X,Y}, \Prob_X \Prob_Y)\nonumber\\
                &=
                \big\Vert\Esub{XY}{k(\cdot,X)\otimes l(\cdot,Y)} - \Esub{X}{k(\cdot,X)}\,\Esub{Y}{l(\cdot,Y)} \big\Vert_{\Hc_{k\otimes l}}^2.\label{HSIC}
                \end{align}
        \end{definition}\label{def:HSIC}
        The name of HSIC stems from the point of view, held by \citet{gretton}. The term within the norm in (\ref{HSIC}) can be identified with the cross-covariance operator \(C_{XY}\colon\Hc_k\to\Hc_l\) for which
        \begin{equation}\label{eq:hsic_cov}
        \langle f, C_{XY} g \rangle_{\Hc_k} = \cov{f(X)}{g(Y)}\quad \forall f\in \Hc_k\,\forall g\in\Hc_l
        \end{equation}
        holds. Consequently, HSIC is the squared Hilbert-Schmidt norm \(\lVert C_{XY}\rVert^2_{\textnormal{HS}}\).\newline
        Coming full circle, we see that using universal kernels \(k\) and \(l\), which causes \(k\otimes l\) to be universal as well, has two important implications. First, Lemma \ref{lem:univ_kernel} states that HSIC is indeed a valid metric to measure dependence between random variables. Second, Definition \ref{def:univ_kernel} yields that \(\Hc_k\) and \(\Hc_l\) are dense in \(C(D_X)\) and \(C(D_Y)\), respectively. Hence, (\ref{eq:hsic_cov}) directly reflects the characterisation of independence given in Proposition \ref{prop:gen_indep}.\newline
        Moreover, it can be shown that Definition 4 and Definition \ref{def:hsic} are equivalent; yet, the former is more convenient to develop estimators.

\subsection{HSIC-Estimation}
        Since the introduction of the Hilbert-Schmidt independence criterion several estimators have been proposed. We assume that a data sample \(\{(x_j, y_j)\}_{j=1}^n\) is given and that the kernels \(k\) and \(l\) are universal and w.l.o.g. bounded by 1.
        \citet{gretton} proposed a simple estimator, which, however, exhibits a bias of order \(\OL(n^{-1})\), whereas \citet{song} corrected this unfavourable trait putting forward an unbiased estimator.
        \begin{definition}\label{def:hsic_estim}
                Let \(K\) and \(L\) be defined by \(K_{ij} = k(x_i,x_j)\) and \(L_{ij} = l(x_i, x_j)\) for \(1\leq i,j \leq n\) and set \(\Kt = K - \text{diag}(K)\), \(\Lt = L - \text{diag}(L)\) and
                \(\Gamma = \Id - \frac{1}{n}\one\one^T\), where \(\one\in \R^n\) has one at every entry. The \emph{biased} and \emph{unbiased HSIC-estimators} \(\HB(X,Y)\) and \(\Hu(X,Y)\) are defined as
                \begin{align*}
                \HB(X,Y) &= \frac{1}{(n-1)^2} \tr(K\Gamma L\Gamma),\\
                \Hu(X,Y) &= \frac{1}{n(n-3)}
                \bigg(\tr(\Kt\Lt) +
                \frac{\one^T\Kt\one\,\one^T\Lt\one}{(n-1)(n-2)}-
                \frac{2}{n-2}\one^T \Kt\Lt\one\bigg).
                \end{align*}
        \end{definition}
        In order to develop and establish properties of estimators, it proves advantageous to use the theory of U-statistics (\citeyear{hoeffding}). This broad class of estimators was pioneered by \citeauthor{hoeffding} and provides a versatile framework to establish useful properties for a multitude of estimators. We use the definition of \cite{lee_u_stats}.
        \begin{definition}\label{def:u-stats}
                Let \(X_1,\dotsc,X_n\) be i.i.d. random variables, which take values in a measurable space \((A, \mathcal{A})\) and share the same distribution, and let \(h\colon A^k\rightarrow \R\) be a symmetric function. We denote \(\Sk_{n,k}\) as the set of all k-subsets of \(\{1,\dotsc,n\}\). For \(n \geq k\),
                \begin{equation*}
                        U_n = \binom{n}{k}^{-1} \sum_{\mathmakebox[30pt][c]{(i_1,\dotsc,i_k)\in\Sk_{n,k}}} h(X_{i_1},\dotsc,X_{i_k})
                \end{equation*}
                is a \emph{U-statistic of degree \(k\) with kernel \(h\)}.
        \end{definition}
        \citet{song} proved that \(\Hu\) indeed has an according representation.
        \begin{theorem}
                Using the notation of Definition \ref{def:hsic_estim}, \(\Hu\) is a U-statistic of degree 4 with kernel
                \begin{equation*}
                h(i,j,q,r)= \frac{1}{24} \sum_{(s,t,u,v)}^{(i,j,q,r)} K_{st}(L_{st}+L_{uv}-2L_{su}).
                \end{equation*}
                The sum is taken over all 24 quadruples \((s,t,u,v)\) that can be selected without replacement from \((i,j,q,r)\) and the notation of \(h\) was reduced to only contain the indices.
        \end{theorem}
\subsection{Lasso Formulation of Normal Weighted HSIC-Lasso}
        We consider a normal weighted HSIC-Lasso selection with associated estimate
        \begin{equation}
                \hat{\beta} = \argmin_{\beta\in\R^p_+} -
                \beta^T H + \frac{1}{2}\,\beta^T M \beta + \lambda\, \beta^T\! w,
                \label{eq:wn_hsic-lasso}
        \end{equation}
        according to Definition 8.\newline
        Assuming that \(M\) is positive definite, we can reformulate (\ref{eq:wn_hsic-lasso}) in terms of a Lasso-problem as follows
        \begin{equation*}\label{lasso_rep}
                \hat{\beta}  = \argmin_{\beta\in\R^p_+}
                \frac{1}{2}\lVert Y-U\beta \rVert_2^2 + \lambda\, \beta^T w.
        \end{equation*}
        \(U\) is determined by the Cholesky decomposition \(M = U^T U\) and \(Y\) is the solution to \(H = U^T Y\). This formulation facilitates the computation of the estimate as there is a variety of efficient algorithms and software packages for Lasso problems available. These are tailored to optimise expressions with a regularisation term and, therefore, yield sparse solutions.

\section{Hypothesis Testing}
        We assume that the response \(Y\) follows a normal distribution \(\N(\mu, \Sigma)\), where \(\mu\) is unknown and \(\Sigma\) is given, and that the selection event can be represented as a polyhedron, i.e. \(\{\hat{S}=S\} = \{AY\leq b\}\). Furthermore, the quantity of interest can be expressed as \(\eta_S^T \mu\), but we drop the dependence on \(S\) in the following. \citet{taylor14} described how one- and two-sided hypothesis testing and confidence interval calculation can be done in this setting. Suppose we want to test
        \begin{equation*}\label{eq:hypothesis_testing}
                \text{H}_0: \eta_S^T \mu = 0
                \quad\text{against}\quad
                \text{H}_1: \eta_S^T \mu > 0.
        \end{equation*}
        Then the statistic
        \begin{equation*}
                T_1 = 1-F_{0,\eta^T\Sigma\eta}^{[\V^-(Z),\V^+(Z)]}(\eta^T Y)
        \end{equation*}
        is a valid p-value for H\textsubscript{0} conditional on \(\{AY\leq b\}\). Further, defining \(\delta_\alpha\) for \(0\leq\alpha\leq1\) such that
        \begin{equation*}
                1-F_{\delta_\alpha,\eta^T\Sigma\eta}^{[\V^-(Z),\V^+(Z)]}(\eta^T Y) = \alpha
        \end{equation*}
        yields a valid one-sided confidence interval \([\delta_\alpha,\infty)\) conditional on \(\{AY\leq b\}\).\newline
        Likewise, we consider the two-sided hypothesis testing problem
        \begin{equation*}
                \text{H}_0: \eta_S^T \mu = 0
                \quad\text{against}\quad
                \text{H}_1: \eta_S^T \mu \neq 0
        \end{equation*}
        and use the statistic
        \begin{equation*}
                T_2 = 2 \min\left\{F_{0,\eta^T\Sigma\eta}^{[\V^-(Z),\V^+(Z)]}(\eta^T Y), 1 - F_{0,\eta^T\Sigma\eta}^{[\V^-(Z),\V^+(Z)]}(\eta^T Y)\right\}.
        \end{equation*}
        Again, \(T_2\) is a valid conditional p-value and defining \(\delta_{\alpha/2}\) and \(\delta_{1-\alpha/2}\) such that
        \begin{align*}
                1-F_{\delta_{\alpha/2},\eta^T\Sigma\eta}^{[\V^-(Z),\V^+(Z)]}(\eta^T Y) &= \alpha/2, \\
                1-F_{\delta_{1-\alpha/2},\eta^T\Sigma\eta}^{[\V^-(Z),\V^+(Z)]}(\eta^T Y) &= 1-\alpha/2
        \end{align*}
        yields a valid confidence interval \([\delta_{\alpha/2}, \delta_{1-\alpha/2}]\) conditional on \(\{AY\leq b\}\).

\section{Proofs}
\subsection{Intermediary Results}
This subsection collects technical results that will be used in the proofs of the following subsections.
First, we state an auxiliary lemma that corresponds to Lemma A in Section 4.3.3 of \cite{lee_u_stats}.
        \begin{lemma}\label{lem:aux}
            Let the random variables \(Z_1,\dots,Z_N\) have a multinomial distribution \(\textnormal{Mult}(m; N^{-1},\dotsc,N^{-1})\) and let \((a_i)_{i\in\Nat}\) be a sequence having the properties \({\lim_{N\to\infty} N^{-1}\sum_{i=1}^{N}a_i =0}\) and \(\lim_{N\to\infty} N^{-1}\sum_{i=1}^{N}a_i^2~=~\sigma^2\). Then,
                \begin{equation*}
                        m^{-\frac{1}{2}}\sum_{i=1}^{N} a_i(Z_i - m/N)\toD \N(0,\sigma^2),\quad\text{as }m,N\to\infty.
                \end{equation*}
        \end{lemma}
\citeauthor{u_stats_rus} presented a multidimensional version of the central limit theorem for U-statistics (\citeyear{u_stats_rus}).
\begin{theorem}\label{th:clt_u_stats}
    Let \(U_n^{(1)},\ldots,U_n^{(m)}\) be U-statistics according to Definition \ref{def:u-stats} and let \(X_1^{(i)},\ldots,X_n^{(i)}, i\in\{1,\ldots,m\},\) be the corresponding i.i.d. random variables. The respective kernels, degrees and expectations are denoted by \(h^{(i)}\), \(k^{(i)}\) and \(\theta^{(i)}\). We introduce the definitions
    \begin{equation*}
        \psi^{(i)}(x) := \E{h^{(i)}(x,X_2^{(i)},\ldots,X_{k^{(i)}}^{(i)}) - \theta^{(i)}},\quad
        \sigma^{(i,j)} := \E{\psi^{(i)}(X_1^{(i)})\,\,\psi^{(j)}(X_1^{(j)})},\quad
        i,j\in\{1,\ldots,m\}.
    \end{equation*}
    If \(\sigma^{(i,i)} > 0\) and \(\E{\big(h^{(i)}(X_1^{(i)},\ldots,X_{k^{(i)}}^{(i)})\big)^2} < \infty\) hold for all \(i\in\{1,\ldots,m\}\), then
    \begin{equation*}
        \sqrt{n}
        \begin{pmatrix}
                (U_n^{(1)}-\theta^{(1)})/k^{(1)} \\
                \vdots \\
                (U_n^{(m)}-\theta^{(m)})/k^{(m)}
                \end{pmatrix}
        \toD
        \N(0,\Sigma),\quad\text{as }n\to \infty,
    \end{equation*}
    where the elements of \(\Sigma\) are given by \(\Sigma_{ij} = \sigma^{(i,j)}, i,j\in\{1,\ldots,m\}\).
\end{theorem}

\subsection{Proof of Theorem \ref{HSIC-CLT}}
        The statement for \(\HvBl\) is a direct consequence of the multidimensional central limit theorem. The expression \(\sqrt{n/B} \big(\HvBl - H_0\big)\) can be written as
        \begin{equation*}
                \sqrt{n/B} \left(\frac{1}{n/B}\sum_{b=1}^{n/B}
                \begin{pmatrix}
                \Hu(X^{b,(1)},Y^b) \\
                \vdots \\
                \Hu(X^{b,(p)},Y^b)
                \end{pmatrix}
                -
                \begin{pmatrix}
                \HSIC(X^{(1)},Y) \\
                \vdots \\
                \HSIC(X^{(p)},Y)
                \end{pmatrix}\right).
        \end{equation*}
        The \(n/B\) random variables in the sum are independent and identically distributed due to the i.i.d. assumption and data subdivision. Moreover, the involved estimators are unbiased and \(n/B \to \infty\).\newline\newline
        In order to prove the second statement of Theorem \ref{HSIC-CLT}, we use an adaptation of the one-dimensional proof of asymptotic normality for an incomplete U-statistics estimator using random subset selection, cf. Theorem 1 in Section 4.3.3 of \cite{lee_u_stats}.
        We prove multidimensional convergence with the Cram\'er-Wold device, see e.g. Theorem 11.2.3 of \cite{bible}. That is it suffices to show that \(\sqrt{m}\,\nu^T\big(\Hvinc - H\big)\) converges to a one-dimensional Gaussian distribution as \(m\to\infty\) for any \(\nu\in\R^p\).\newline
        We introduce the independent random vectors \(Z^{(j)}, j\in\{1,\dotsc,p\}\) and index their entries with \(\Sk_{n,4}\); hence, their elements are \(\{Z_S^{(j)}\colon S\in\Sk_{n,4}\}\). All of them follow a multinomial distribution \(\textnormal{Mult}(m; N^{-1},\dotsc,N^{-1})\), where \(N = \binom{n}{4}\). Hence, we can write
        \begin{equation}\label{eq:proof1}
                m^\frac{1}{2}\,\nu^T\big(\Hvinc - H\big) =
                m^{-\frac{1}{2}}\,\nu^T\sum_{S\in\Sk_{n,4}} Z_S \big(h(S) - H\big),
        \end{equation}
        where the sum as well as the product within is to be understood componentwise, and \({Z=(Z^{(1)},\dots,Z^{(p)})}\) as well as \(h\) are used in a vectorised way, slightly abusing notation. In order to derive the asymptotic distribution of (\ref{eq:proof1}), we consider its characteristic function \(\phi_n\).
        In the following manipulations we drop the indices for the summation \(\sum_{S\in\Sk_{n,4}}\), introduce the notation \(\cramped{X^{(j)} = (X_1^{(j)},\ldots,X_n^{(j)}), j\in\{1,\ldots,p\}}\), and \(Y\) accordingly, and denote the \(p\)-dimensional vector of (complete) U-statistics by \(U_n\), that is the vector of unbiased HSIC-estimators. Then:
        \begin{align*}
                \phi_n(t) &=
                \E{\exp\left(it\,m^{-\frac{1}{2}}\,\nu^T\sum Z_S \big(h(S) - H\big)\right)}\\
                &=
                \E{\Ec{\exp\left(it\,m^{-\frac{1}{2}}\,\nu^T\sum Z_S \big(h(S) - H\big)\right)}{X^{(1)},\ldots,X^{(p)}, Y}}\\
                &= \E{\exp\bigg(it\,m^\frac{1}{2}\!\sum_{j=1}^{p} \nu_j (U_n^{(j)}\!-\!H_j)\bigg) \Ec{\exp\bigg(it\,m^{-\frac{1}{2}}\!\sum_{j=1}^{p}\!\nu_j\sum \Big(Z_S^{(j)}\! - \frac{m}{N}\Big) \big(h_j(S)\! -\! H_j\big)\bigg)}{X^{(1)},\hbox to 1em{.\hss.\hss.},X^{(p)}\!, Y}}\\
                &= \E{\exp\bigg(it\,m^\frac{1}{2}\sum_{j=1}^{p} \nu_j (U_n^{(j)}\!-\!H_j)\bigg)\Ec{\exp\left(it\,m^{-\frac{1}{2}}\nu_j\sum \Big(Z_S^{(j)}\! - \frac{m}{N}\Big) \big(h_j(S)\! -\! H_j\big)\right)}{X^{(1)},\hbox to 1em{.\hss.\hss.},X^{(p)}\!, Y}}.
        \end{align*}
        In the manipulations above we used the tower law of conditional expectation and the independence of the \(Z_S^{(j)}, {j\in\{1,\ldots,p\}}\). Moreover, we inserted \(\pm \,m(U_n-H)=m/N \sum (h(S)-H)\).\newline
        Having separated the randomness coming from the data and the subset selection, we treat the second factor in the product above. Standard U-statistics theory implies that
        \begin{equation*}
                \lim_{N\to\infty} N^{-1} \sum_{S\in\Sk_{n,4}} (h_j(S)-H_j) = 0\quad\text{and}\quad
                \lim_{N\to\infty} N^{-1} \sum_{S\in\Sk_{n,4}} (h_j(S)-H_j)^2 = \sigma^2_j
        \end{equation*}
        almost surely, where \citet{song} stated a formula for \(\sigma^2_j\). Ergo, the requirements of Lemma \ref{lem:aux} are fulfilled and applying it together with  the dominated convergence theorem yields
        \begin{align*}
                \lim_{n\to\infty} \phi_n(t) &= \lim_{n\to\infty} \E{\exp\left(it\,m^\frac{1}{2}\sumP \nu_j (U_n^{(j)}\!-\!H_j)\right)} \prod_{j=1}^p \exp\big(-(\sigma_j\nu_j)^2t^2/2\big) \\
                &= \lim_{n\to\infty}\E{\exp\left(it\sqrt{m/n}\, \nu^T\sqrt{n}\,(U_n-H)\right)} \prod_{j=1}^p \exp\big(-(\sigma_j\nu_j)^2t^2/2\big).
        \end{align*}
        \citet{aistats} pointed out that \(\sigma^{(i,i)}=0\), according Theorem \ref{th:clt_u_stats}, holds for \(Y\independent X^{(i)}\), whereas \(\sigma^{(i,i)}>0\) is true if the response and the \(i\)-th covariate are dependent. Therefore, we define the index set \({I:=\{i\colon \HSIC(X^{(i)},Y) > 0\}}\) and the positive definite matrix \(\Xi\) by \(\Xi_{ij} = 16\,\sigma^{(i,j)}, i,j\in I\).        Using Theorem \ref{th:clt_u_stats} and Slutsky's theorem, we arrive at
        \begin{equation*}
                \lim_{n\to\infty} \phi_n(t) = \exp\big(-(\sqrt{l}\,\nu_I^T\,\Xi\,\nu_I)\, t^2 /2\big) \prod_{j=1}^p \exp\big(-(\sigma_j\nu_j)^2t^2/2\big).
        \end{equation*}
        The limit of \(\phi_n\) is clearly a Gaussian characteristic function which proves asymptotic normality.\qed

\subsection{Proof of Theorem \ref{th:asymp_pivot}}

We prove Theorem \ref{th:asymp_pivot} in two steps: First, we establish \((H_n, M_n, \Shat_n)\vert\{A_n H_n\leq b_n\} \to (H, M, \Sigma)\vert\{A H\leq b\}\) in distribution; second, we apply the continuous mapping theorem (CMT), see for example Theorem 2.3 (i) in \cite{van_der_vaart},  to the cdf of a truncated Gaussian. The ideas of this proof are heavily influenced by \cite{high_dim_gen1}.\newline\newline
Let \(S\) be a set of selected covariates and \(\tilde{S}\subseteq S\). In the following, we use the abbreviations \(\eta_n =\eta_\St(M_n)\), \(A_n = A_\St(M_n)\) and \(b_n = b_\St(M_n)\). Applying Theorem 2.7 (v) of \cite{van_der_vaart}, we get \((H_n,M_n)\to (H,M)\) in distribution, where \(H\) has the law \(\N(\mu, \Sigma)\). Furthermore, due to the independence of \(\Shat_n\) from \((H_n, M_n)\) we can easily extend the convergence to \((H_n,M_n, \Shat_n)\to (H,M, \Sigma)\). Ultimately, since \(A\) and \(b\) are almost surely continuous, the CMT yields \((H_n,M_n, \Shat_n, A_n H_n\!-b_n)\to (H,M, \Sigma, AH-b)\) in distribution and we define \(\Gamma_n :=(H_n,M_n, \Shat_n)\).\newline
We arrange the components of \(\Gamma_n\) in a vector, fix an arbitrary \(x\in\R^{p+2p^2}\) and analyse the asymptotics of the conditional distribution \((\Gamma_n\vert A_nH_n\leq b_n)\)
\begin{align*}
        \Prob\,\big(\Gamma_n \leq x\,\vert\, A_nH_n\!-b_n \leq 0\big) &=
        \frac{\Prob\,\big(\Gamma_n \leq x, A_nH_n\!-b_n \leq 0\big)}{\Prob\,\big(A_nH_n\!-b_n \leq 0\big)}\\
        &\to
        \frac{\Prob\,\big(\Gamma \leq x, AH-b\leq 0\big)}{\Prob\,\big(AH-b \leq 0\big)}=
        \Prob\,\big(\Gamma \leq x\,\vert\, AH-b \leq 0\big),\quad\text{as }n\to\infty.
\end{align*}
This is true because both the numerator and denominator converge to the respective probabilities and the denominator is bounded away from zero as the interior of \(\{AH\leq b\}\) is not empty. Thus, we have shown that
\begin{equation}\label{eq:asymp_cond}
        (H_n, M_n, \Shat_n)\vert\{A_n H_n\leq b_n\} \toD (H, M, \Sigma)\vert\{A H\leq b\},\quad\text{as }n\to\infty.
\end{equation}
In order to use this convergence result for
\begin{equation}\label{eq:asymp_F}
        F_{\eta_n^T\mu,\eta_n^T\Shat_n\eta_n}^{[\V^-(Z_n),\V^+(Z_n)]}(\eta_n^T H_n)\vert \{A_n H_n\leq b_n\},
\end{equation}
where \(Z_n\) is defined according to Lemma \ref{lem:polyhedral_lemma}, we have to verify that this expression is an a.s. continuous function of \((H_n, M_n, \Shat_n)\). The cdf of a truncated Gaussian random variable \(F_{x_1, x_2}^{[x_4, x_5]}(x_3)\) has five arguments and is continuous, if \(x_4 < x_5\) holds.\newline
In (\ref{eq:asymp_F}), we plug in \(\eta_n^T\mu\), \(\eta_n^T\Shat_n\eta_n\) and \(\eta_n^T H_n\) for \(x_1\), \(x_2\) and \(x_3\), respectively, which are, by assumption, a.s. continuous with respect to \(\Gamma_n\). The truncation points \(\V^-(Z_n)\) and \(\V^+(Z_n)\) are the maximum and minimum of finite sets of continuous functions. We denote these two sets \(G^-\) and \(G^+\). Hence, all discontinuity points of \(\V^-(Z_n)\) and \(\V^+(Z_n)\) are contained in \(E = \bigcup_{j=1}^p\{\e_j^TA_nC_n=0\}\), i.e. the set of points where the functions contained in \(G^-\) and/or \(G^+\) change.
As a finite union of lower-dimensional subspaces, \(E\) is a null set. Therefore, \(\V^-(Z_n)\) and \(\V^+(Z_n)\) are almost surely continuous functions of \(\Gamma_n\). Moreover, we deduce from the definition of the truncation points that
\begin{equation*}
        \V^-(Z) = \V^+(Z)\quad\Leftrightarrow\quad
        \eta^TH = \frac{b_j-(AZ)_j}{(AC)_j}\quad\forall j\in J,
\end{equation*}
where \(J\) is defined as \(\{j\colon(AC)_j\neq 0\}\). Rearranging these equations, we arrive at \(\{\V^-(Z) = \V^+(Z)\} = {\{(AH)_j = b_j\,\forall j\in J\}}\). As a lower-dimensional subspace this set has measure zero and, consequently, \({\V^-(Z) < \V^+(Z)}\) holds almost surely. In summary, (\ref{eq:asymp_F}) depends on \((H_n, M_n, \Shat_n)\) in an a.s. continuous fashion. Using (\ref{eq:asymp_cond}), the CMT and Theorem \ref{theo:polyhedral_lemma}, we obtain
\begin{equation*}
        F_{\eta_n^T\mu,\eta_n^T\Shat_n\eta_n}^{[\V^-(Z_n),\V^+(Z_n)]}(\eta_n^T H_n)\vert \{A_n H_n\leq b_n\} \toD
        F_{\eta^T\mu,\eta^T\Sigma\eta}^{[\V^-(Z),\V^+(Z)]}(\eta^T H)\vert \{A H\leq b\}
        \sim
        \mathrm{Unif}\,(0,1).
\end{equation*}
\qed


\subsection{Proof of Theorem \ref{th:trunc_points}}
        In order to characterise the selection event \(\{\Sh = S\}\), we assume w.l.o.g. that the first \(\lvert S\rvert\) covariates of \(\{X_1,\ldots,X_p\}\) were included into the model. We rely on the Karush-Kuhn-Tucker (KKT) conditions, cf. Section 5.5.3 of \cite{boyd}, that identify the solution of an optimisation problem by a set of equations and inequalities. Since the function to be minimised is convex due to the positive definiteness of \(M\) and Slater's condition, cf. Section 5.2.3 of \cite{boyd}, the KKT conditions provide an equivalent characterisation of the solution of the HSIC-Lasso problem.
        We obtain
        \begin{equation}\label{eq:KKT}
        \begin{gathered}
                0 = - H + M\hat{\beta} + \lambda w - u, \\
                \beta_j \geq 0,\qquad u_j \geq 0,\qquad \beta_j u_j = 0,\qquad \forall j\in\{1,\ldots,p\}.
        \end{gathered}
        \end{equation}
        We partition the upper inequalities along \(S\) and \(S^c\) and get
        \begin{align}
                \hat{\beta}_S &= M_{SS}^{-1}(H_S - \lambda\, w_S),\label{eq:1} \\
                0 &\leq H_{S^c} + (M\hat{\beta})_{S^c} - \lambda\, w_{S^c}.\label{eq:2}
        \end{align}
        These results translate into two set of inequalities. First, all entries of \(\hat{\beta}\) must be non-negative which implies
        \begin{equation*}
                0 \leq M_{SS}^{-1}(H_S - \lambda\, w_S)
                \quad\Leftrightarrow\quad
                -\lambda^{-1}\left(M_{SS}^{-1} \,\vert\, 0\right) H \leq
                -M_{SS}^{-1}\,w_S.
        \end{equation*}
        Second, \(M\hat{\beta} = M_S\hat{\beta}_S\) holds by definition of \(\Sh\). Hence, we can plug (\ref{eq:1}) into (\ref{eq:2}) and obtain
        \begin{align*}
                &0 \leq H_{S^c} + M_{SS^c}\big(M_{SS}^{-1}(H_S - \lambda\, w_S)\big) - \lambda\, w_{S^c} \\
                \Leftrightarrow\quad &-\lambda^{-1}\left(M_{SS^c}M_{SS}^{-1} \,\vert\, \Id\right) H
                \leq
                w_{S^c}-M_{SS^c}M_{SS}^{-1}\,w_S.
        \end{align*}
        Both these set of inequalities describe the selection in an affine linear fashion. In this setting, we can use the polyhedral lemma to compute the truncation points \(\V^-\) and \(\V^+\).
        \newline\newline
        For the selection event \(\{j \in\Sh\}\), we again use the KKT conditions (\ref{eq:KKT}). For any \(j\in S\), \(u_j\) equals zero and we can express \(H_j\) as follows
        \begin{equation*}
                H_j = \e_j^T\left(M\bh + \lambda w\right) =
                \e_j^T\left(M\bh_{-j} + \lambda w\right) + M_{jj}\bh_j
                > \e_j^T\left(M\bh_{-j} + \lambda w\right),
        \end{equation*}
        where \(\bh_{-j}\!\) denotes \(\bh\) with the \(j\)-th entry set to zero. This estimation holds true as \(\bh_j\) is positive by definition of \(\Sh\) and \(M_{jj} = \e_j^T\! M \e_j > 0\) because \(M\) is positive definite. Rearranging the inequality, we obtain
        \begin{equation*}
                -\e_j^T H < -\,\e_j^T M \bh_{-j}\! - \lambda w_j.
        \end{equation*}
        \qed
        \newline
        \begin{remark}
                Since the selection event \(\{j \in\Sh\}\) is less complex than \(\{\Sh = S\}\), it is possible to directly derive the truncation points without the need for the polyhedral lemma.\newline
                To this end, we decompose \(H\) into a component in direction of \(\eta\) and one perpendicular to \(\eta\)
                \begin{equation*}
                        H = (\eta^T H) \cdot C + Z.
                \end{equation*}
                Again, we apply the KKT conditions (\ref{eq:KKT}) and obtain
                \begin{equation*}
                        0 = (\eta^T H) \cdot C - Z + M\hat{\beta} + \lambda w - u,
                \end{equation*}
                with \(u\in\R^p_+\). Since \(j\notin\hat{S} \Leftrightarrow\hat{\beta}_j = 0\) holds by definition of \(\hat{S}\), the inequality
                \begin{equation}\label{eq:3}
                        0 \leq \e_j^T \left[(\eta^T H) \cdot C - Z + M\hat{\beta}_{-j} + \lambda w\,\right],
                \end{equation}
                ensues for this case. Rearranging (\ref{eq:3}), we find
                \begin{equation}
                        \eta^T H \leq \frac{1}{\e_j^T C}\left[\e_j^T M\hat{\beta}_{-j} - \e_j^T Z + \lambda w_j\right]. \label{eq:V-}
                \end{equation}
                Consequently, for the event \(\{j\in\hat{S}\}\) the lower truncation point \(\V^-(Z)\) is the RHS of (\ref{eq:V-}) and \(\V^+(Z) = \infty\).
        \end{remark}

        \section{Pseudocode of the Algorithm}
        Along with the description of the algorithm in Section \ref{sec:algorithm}, we give a more detailed account on the different steps of our PSI-procedure for HSIC-Lasso in the following.
        \begin{algorithm}[h!]
        \caption{Post-selection inference for HSIC-Lasso selection with HSIC- or partial target}
        \label{alg:alg}
        \begin{algorithmic}
            \STATE \textbf{Input:} data $(\Xd^n, \Yd^n)$; level $\alpha$; inference target $t$; split ratio $s$; number of screened variables $P$; screen-, $M$- and $H$-estimators $\text{e}_s, \text{e}_M, \text{e}_H$
            \STATE \textbf{Output:} significant variables $I_{sig}$

            \STATE$(\Xd^{n,1}, \Yd^{n,1}), (\Xd^{n,2}, \Yd^{n,2}) \leftarrow \text{split}((\Xd^n, \Yd^n),s)$

            \hfill\COMMENT{1st fold}
            \STATE $H^{(1)} \leftarrow \text{estimate}_H(\Xd^{n,1}, \Yd^{n,1}, \text{e}_s)$
            \STATE $I_{sc} \leftarrow \text{screening}(H^{(1)}, P) $
            \STATE$M^{(1)} \leftarrow \text{estimate}_M(\Xd^{n,1}_{I_{sc}}, \text{e}_s)$
            \STATE $\tilde{M}^{(1)}\leftarrow\text{positive-definite-approximation}(M^{(1)})$
            \STATE $U_1 \leftarrow \text{cholesky}(\tilde{M}^{(1)});\,\, Y_1 \leftarrow U^{-T}_1 H^{(1)}_{I_{sc}}$
            \STATE $\lambda \leftarrow \text{cross-validation}(U_1, Y_1)$ \hfill\COMMENT{or AIC}
            \STATE $w \leftarrow \text{weights}(U_1, Y_1)$

            \hfill\COMMENT{2nd fold}
            \STATE $H^{(2)} \leftarrow \text{estimate}_H(\Xd^{n,2}_{I_{sc}}, \Yd^{n,2}_{I_{sc}}, \text{e}_H)$
            \STATE $M^{(2)} \leftarrow \text{estimate}_M(\Xd^{n,2}_{I_{sc}}, \text{e}_M)$
            \STATE $\tilde{M}^{(2)}\leftarrow\text{positive-definite-approximation}(M^{(2)})$
            \STATE $\hat{\Sigma} \leftarrow \text{estimate}_{\Sigma}(H^{(2)})$
            \STATE $U_2 \leftarrow \text{cholesky}(\tilde{M}^{(2)});\,\,Y_2 \leftarrow U^{-T}_2 H^{(2)}$
            \STATE $\hat{\beta} \leftarrow \text{lasso-opimisation}(Y_2, U_2, \lambda, w)$
            \STATE $S \leftarrow \text{non-zero-indices}(\hat{\beta})$
            \STATE $I_{sig} \leftarrow \emptyset$

            \IF{$t$ is partial target}
                \STATE $A \leftarrow -\lambda^{-1}
                \begin{pmatrix}
                        (\tilde{M}^{(2)}_{SS})^{-1} &\vert& 0\,\\
                        \tilde{M}^{(2)}_{S^cS}(\tilde{M}^{(2)}_{SS})^{-1}&\vert& \Id\,
                \end{pmatrix};\quad
                b \leftarrow
                \begin{pmatrix}
                        -(\tilde{M}^{(2)}_{SS})^{-1}\,w_S\\
                        w_{S^c}-\tilde{M}^{(2)}_{S^cS}(\tilde{M}^{(2)}_{SS})^{-1}\,w_S
                \end{pmatrix}
                $
                \ENDIF

                \FORALL{$j \in S$}
                \IF{$t$ is HSIC-target}
                        \STATE $\eta \leftarrow \e_j$
                        \STATE $C \leftarrow (\eta^T\hat{\Sigma}\eta)^{-1}\hat{\Sigma}\,\eta;\quad Z \leftarrow(\Id - C \eta^T)H^{(2)}$
                        \STATE $\V^-_j \leftarrow (\e_j^T C)^{-1}\left[\e_j^T \tilde{M}^{(2)}\hat{\beta}_{-j} - \e_j^T Z + \lambda w_j\right]; \quad\V^+_j \leftarrow \infty$
                \ELSIF{$t$ is partial target}
                        \STATE $\eta \leftarrow \e_j^T ((\tilde{M}_{SS}^{(2)})^{-1}\,\vert\,0)$
                        \STATE $\V^-, \V^+ \leftarrow \text{truncation-points}(A, b, \eta,\hat{\Sigma})$
                \ENDIF
                \STATE $p \leftarrow 1-F_{0,\eta^T\hat{\Sigma}\eta}^{[\V^-,\V^+]}(\eta^T H)$
                \IF{$p \leq \alpha$}
                \STATE $I_{sig} \leftarrow I_{sig} \cup \{j\}$
                \ENDIF
                \ENDFOR
            \STATE \textbf{return} $I_{sig}$
        \end{algorithmic}
        \end{algorithm}

        \section{Additional Experiment}
        We analyse the \textit{Communities and Crime} datset from the UCI Repository, cf. \cite{uci}, which was created by \cite{redmond} and combines socio-economic (1990), law enforcement (1990) and crime data (1995) from 1\,994 different US communities.\footnote{The sources of the dataset are \cite{data1}, \cite{data2}, \cite{data3} and \cite{data4}.} The authors provide 122 (numerical) features to predict the total number of violent crimes per 100\,000 inhabitants. Since the law enforcement data is complete for only 319 communities, we carry out two different analyses: First, we examine the subset of 319 datapoints with complete features; then, we consider the whole dataset but only use covariates without missing data. (Moreover, we delete one data point with a missing value in a feature that is not part of the law enforcement data.)\newline
        We highlight that our approach does not target causal relationships; instead, it merely concerns the associations between a feature and the absence or presence of violent crimes. Moreover, as we rely on the Hilbert-Schmidt independence criterion, the existence of a relationship only but not its strength or direction is captured. Lastly, we emphasise that our analysis of the dataset focuses on the characteristics of the proposed methods, rather than on drawing sociological conclusions. To the latter end, a more in-depth analysis is necessary.
        \newline
        Since the number of covariates is manageable, we do not screen features but use 25 \% and 20 \% of the data, respectively, to determine the regularisation parameter of the non-adaptive HSIC-Lasso via 10-fold cross-validation. We employ the Gaussian kernel throughout as the data is continuous and use the unbiased HSIC-estimator for the matrix $M$. Both the block and the incomplete estimators are examined, each with different sizes, the number of selected variables for Multi is set to $k=10$ and we use the confidence level $\alpha = 0.05$.\newline
        We depict the findings of the two different analyses in Table \ref{tab:all-features} and \ref{tab:restricted-features}, respectively. First, we notice that selection via HSIC-Lasso effectively reduces the number of highly correlated features among the selected covariates compared to HSIC-ordering. For instance, the latter chooses the features \textit{racePctBlack} and \textit{racePctWhite}, which describe the percentage of the population which is African American and Caucasian, respectively, and are presumably highly correlated, for every applied HSIC-estimator in both analyses. On the contrary, HSIC-Lasso almost always only selects \textit{racePctWhite}. This behaviour becomes even more evident when we consider the features \textit{pctFam2Par}, \textit{pctKids2Par}, \textit{pctYoungKids2Par} and \textit{pctTeen2Par} which denote the percentage of two parent families with children and children, young children and teenagers in family housing with two parents, respectively. Due to the strong association of these covariates, HSIC-ordering regularly selects all four of them whereas HSIC-Lasso mostly chooses only one of them as they are highly dependent.\newline
        In both analyses, we notice that there is only moderate disagreement among the different HSIC-estimators in terms of the selected features; though, it is more pronounced in the first analysis as the sample size is lower and the number of covariates higher.
        Moreover, we observe that the incomplete HSIC-estimator, in particular with larger sizes, yields a higher number of accepted covariates than the block estimator. We attribute this behaviour to the fact that the accuracy of the estimates grows with the sizes $l$ and $B$, respectively; however, we cannot choose the block size too large in order to preserve the normality assumption whereas there exists no such restriction for the incomplete HSCI-estimator. For this reason, we argue to particularly focus on the results in the I20 columns.\newline
        Furthermore, the \textit{Communities and Crime} dataset highlights the utility of the partial target and thus the benefits of using post-selection inference with HSIC-Lasso instead of HSIC-ordering. Contrary to the HSIC-target which only considers the association between a feature and the outcome, the partial target adjusts for the dependence structure among all selected covariates. In both analyses, this leads to a small set of accepted features that presumably are not dependent on each other to a large degree. Hence, researchers obtain a clearer and more interpretable view on the data compared to using the HSIC-target.\newline
        Lastly, we compare the findings of the two different analyses mostly focusing on the results of the partial target for the incomplete HSIC-estimator with size $l=20$. Some features like \textit{racePctWhite}, \textit{pctKids2Par} and  \textit{pctIlleg}, the percent of children born to never married, are agreed upon by both studies. On the other hand, \textit{numStreet}, the number of homeless people counted in the street, \textit{PctVacantBoarded}, the percentage of vacant housing that is boarded up, and \textit{PolicCars}, the number of police cars, are only found significant in the first analysis whereas \textit{pctWInvInc}, the percentage of households with investment/rent income in 1989, \textit{FemalePctDiv}, the percentage of divorced females, \textit{PctHousLess3BR}, the percentage of housing units with less than 3 bedrooms, and \textit{LemasPctOfficDrugUn}, the percentage of officers assigned to drug units, are only signficant in the second analysis. This hints that there might be correlation between features and the presence/absence of missing data stemming from the data collection and/or assimilation giving rise to the discrepancies between the two analyses. Yet, a more in-depth analysis is necessary to dissolve this uncertainty.

        \setlength{\arrayrulewidth}{0.2mm}
    \setlength{\tabcolsep}{3pt}
    \renewcommand{\arraystretch}{1.3}
        \begin{table}[t]
        \caption{Acceptance of the HSIC- and partial target for the subset of data points without missing values ($n=319, p=122$). We denote the feature names according to their abbreviation in the UCI Repository and the column names describe the different HSIC estimators and their sizes, e.g. I10 is the shorthand notation for the incomplete HSIC-estimator of size 10. Grey rectangles denote that a feature was selected but rejected to be significant at level $\alpha=0.05$, whereas a black rectangle means acceptance and significance.}
        \label{tab:all-features}
        \vskip 0.1in
        \begin{center}
    \begin{small}
    \begin{sc}
        \begin{tabular}{|c|*{5}{>{\centering\arraybackslash}p{0.5cm}}|*{5}{>{\centering\arraybackslash}p{0.5cm}}|*{5}{>{\centering\arraybackslash}p{0.5cm}}|}
                \hline
                \multicolumn{1}{|c|}{Feature} & \multicolumn{5}{c|}{Proposal - Partial}
                & \multicolumn{5}{c|}{Proposal - HSIC} & \multicolumn{5}{c|}{Multi - HSIC} \\ \cline{2-16}
                & B5 & B10 & I5 & I10 & I20 & B5 & B10 & I5 & I10 & I20 & B5 & B10 & I5 & I10 & I20 \\\hline
                racePctBlack & \ns & \ns & \sel & \sel & \ns & \ns & \ns & \ac & \ac & \ns & \ac & \ac & \ac & \ac & \ac \\
                racePctWhite & \ns & \sel & \sel & \ac & \ac & \ns & \ac & \ac & \ac & \ac & \ac & \ac & \ac & \ac & \ac \\
                medIncome & \ns & \sel & \ns & \ns & \ns & \ns & \ac & \ns & \ns & \ns & \ns & \ns & \ns & \ns & \ns \\
                pctWWage & \sel & \ns & \ns & \ns & \ns & \sel & \ns & \ns & \ns & \ns & \ns & \ns & \ns & \ns & \ns \\
                pctWFarmSelf & \sel & \ns & \ns & \ns & \ns & \sel & \ns & \ns & \ns & \ns & \ns & \ns & \ns & \ns & \ns \\
                pctWInvInc & \ns & \ns & \sel & \sel & \sel & \ns & \ns & \ac & \ac & \ac & \ns & \ns & \ns & \ns & \ns \\
                pctWPubAsst & \ns & \ns & \sel & \sel & \ns & \ns & \ns & \ac & \ac & \ns & \sel & \ac & \ac & \ac & \ns \\
                perCapInc & \ns & \sel & \ns & \ns & \ns & \ns & \ac & \ns & \ns & \ns & \ns & \ns & \ns & \ns & \ns \\
                blackPerCap & \ns & \sel & \ns & \ns & \ns & \ns & \sel & \ns & \ns & \ns & \ns & \ns & \ns & \ns & \ns \\
                indianPerCap & \sel & \ns & \ns & \ns & \ns & \sel & \ns & \ns & \ns & \ns & \ns & \ns & \ns & \ns & \ns \\
                NumUnderPov & \ns & \ns & \ns & \ns & \sel & \ns & \ns & \ns & \ns & \ac & \sel & \ns & \ns & \ns & \ns \\
                PctPopUnderPov & \ns & \ns & \ns & \ns & \sel & \ns & \ns & \ns & \ns & \ac & \ns & \ac & \ac & \ac & \sel \\
                MalePctDivorce & \ns & \ns & \ns & \ns & \sel & \ns & \ns & \ns & \ns & \ac & \ns & \ns & \ns & \ns & \ns \\
                FemalePctDiv & \ns & \ns & \ns & \sel & \ns & \ns & \ns & \ns & \ac & \ns & \ns & \ns & \ns & \ns & \ns \\
                TotalPctDiv & \ns & \ns & \sel & \ns & \sel & \ns & \ns & \ac & \ns & \ac & \ns & \ns & \ns & \ns & \ns \\
                PctFam2Par & \ns & \ns & \sel & \ns & \ns & \ns & \ns & \ac & \ns & \ns & \sel & \ac & \ac & \ac & \ac \\
                PctKids2Par & \sel & \sel & \sel & \sel & \ac & \ac & \ac & \ac & \ac & \ac & \ac & \ac & \ac & \ac & \ac \\
                PctYoungKids2Par & \ns & \ns & \ns & \sel & \ns & \ns & \ns & \ns & \ac & \ns & \sel & \ac & \ac & \ac & \ac \\
                PctTeen2Par & \ns & \ns & \ns & \ns & \ns & \ns & \ns & \ns & \ns & \ns & \ns & \ac & \ac & \ac & \sel \\
                PctWorkMom & \ns & \sel & \ns & \ns & \ns & \ns & \sel & \ns & \ns & \ns & \ns & \ns & \ns & \ns & \ns \\
                NumIlleg & \ns & \ns & \ns & \ns & \ns & \ns & \ns & \ns & \ns & \ns & \sel & \ns & \ns & \ns & \ns \\
                PctIlleg & \sel & \sel & \sel & \sel & \ac & \ac & \ac & \ac & \ac & \ac & \ac & \ac & \ac & \ac & \ac \\
                MedNumBR & \sel & \ns & \sel & \sel & \ns & \sel & \ns & \sel & \sel & \ns & \ns & \ns & \ns & \ns & \ns \\
                HousVacant & \ns & \ns & \ns & \ns & \ac & \ns & \ns & \ns & \ns & \ac & \ns & \ns & \ns & \ns & \ns \\
                PctVacantBoarded & \sel & \sel & \ac & \ac & \ac & \ac & \ac & \ac & \ac & \ac & \ns & \ac & \ac & \ac & \ac \\
                MedYrHousbuilt & \sel & \ns & \ns & \ns & \ns & \sel & \ns & \ns & \ns & \ns & \ns & \ns & \ns & \ns & \ns \\
                RentHighQ & \ns & \ns & \sel & \sel & \ns & \ns & \ns & \ac & \ac & \ns & \ns & \ns & \ns & \ns & \ns \\
                NumInShelters & \ns & \ns & \ns & \ns & \sel & \ns & \ns & \ns & \ns & \ac & \ns & \ns & \ns & \ns & \ns \\
                NumStreet & \ac & \sel & \ac & \ac & \ac & \ac & \ac & \ac & \ac & \ac & \ns & \ns & \ns & \ns & \ns \\
                PctBornSameState & \ns & \ns & \sel & \ns & \ns & \ns & \ns & \ac & \ns & \ns & \ns & \ns & \ns & \ns & \ns \\
                LemasTotalReq & \sel & \sel & \ns & \ns & \ns & \sel & \sel & \ns & \ns & \ns & \ns & \ns & \ns & \ns & \ns \\
                LemasTotReqPerPop & \sel & \sel & \sel & \sel & \ns & \sel & \sel & \ac & \ac & \ns & \ns & \ns & \ns & \ns & \ns \\
                PolicReqPerOffic & \sel & \sel & \ns & \ns & \sel & \sel & \sel & \ns & \ns & \ac & \ns & \ns & \ns & \ns & \ns \\
                RacialMatchCommPol & \ns & \ns & \sel & \sel & \ns & \ns & \ns & \ac & \ac & \ns & \ns & \ns & \ns & \ns & \ns \\
                PctPolicWhite & \ns & \sel & \ns & \ns & \sel & \ns & \ac & \ns & \ns & \ac & \ns & \ns & \ns & \ns & \ns \\
                PctPolicBlack & \sel & \ns & \ns & \ns & \ns & \ac & \ns & \ns & \ns & \ns & \ns & \ns & \ns & \ns & \ac \\
                PctPolicHisp & \sel & \ns & \ns & \ns & \ns & \sel & \ns & \ns & \ns & \ns & \ns & \ns & \ns & \ns & \ns \\
                PctPolicMinor & \sel & \sel & \ns & \ns & \ns & \ac & \ac & \ns & \ns & \ns & \ns & \ns & \ns & \ns & \ns \\
                NumKindsDrugsSeiz & \sel & \sel & \ns & \ns & \ns & \sel & \sel & \ns & \ns & \ns & \ns & \ns & \ns & \ns & \ns \\
                PolicAveOTWorked & \ns & \sel & \ns & \ns & \ns & \ns & \sel & \ns & \ns & \ns & \ns & \ns & \ns & \ns & \ns \\
                PctUsePubTrans & \sel & \ns & \ns & \ns & \ns & \ac & \ns & \ns & \ns & \ns & \sel & \ns & \ns & \ns & \ns \\
                PolicCars & \ns & \ns & \sel & \ac & \ac & \ns & \ns & \ac & \ac & \ac & \ns & \ns & \ns & \ns & \ns \\
                LemasPctPolicOnPatr & \ns & \ns & \ns & \ns & \sel & \ns & \ns & \ns & \ns & \sel & \ns & \ns & \ns & \ns & \ns \\
                LemasGangUnitDeploy & \sel & \sel & \ns & \ns & \ns & \sel & \sel & \ns & \ns & \ns & \ns & \ns & \ns & \ns & \ns \\
                LemasPctOfficDrugUn & \sel & \sel & \sel & \sel & \ns & \sel & \sel & \sel & \sel & \ns & \ns & \ns & \ns & \ns & \ns \\\hline
        \end{tabular}
    \end{sc}
    \end{small}
        \end{center}
        \vskip -0.1in
    \end{table}

    \begin{table}[t]
        \caption{Acceptance of the HSIC- and partial target for the subset of data that only contains features without missing data ($n=1993, p=100$). We denote the feature names according to their abbreviation in the UCI Repository and the column names describe the different HSIC estimators and their sizes, e.g. I10 is the shorthand notation for the incomplete HSIC-estimator of size 10. Grey rectangles denote that a feature was selected but rejected to be significant at level $\alpha=0.05$, whereas a black rectangle means acceptance and significance.}
        \label{tab:restricted-features}
        \vskip 0.1in
        \begin{center}
    \begin{small}
    \begin{sc}
        \begin{tabular}{|c|*{6}{>{\centering\arraybackslash}p{0.5cm}}|*{6}{>{\centering\arraybackslash}p{0.5cm}}|*{6}{>{\centering\arraybackslash}p{0.5cm}}|}
                \hline
                \multicolumn{1}{|c|}{Feature} & \multicolumn{6}{c|}{Proposal - Partial}
                & \multicolumn{6}{c|}{Proposal - HSIC} & \multicolumn{6}{c|}{Multi - HSIC} \\ \cline{2-19}
                & B5 & B10 & B20 & I5 & I10 & I20 & B5 & B10 & B20 & I5 & I10 & I20 & B5 & B10 & B20 & I5 & I10 & I20 \\\hline
                racePctBlack & \ns & \ns & \ns & \ns & \ns & \ns & \ns & \ns & \ns & \ns & \ns & \ns & \ac & \ac & \ac & \ac & \ac & \ac\\
                racePctWhite & \sel & \sel & \sel & \sel & \ac & \ac & \ac & \ac & \ac & \ac & \ac & \ac & \ac & \ac & \ac & \ac & \ac & \ac\\
                racePctHisp & \sel & \sel & \sel & \sel & \sel & \ac & \ac & \ac & \ac & \ac & \ac & \ac & \ns & \ns & \ns & \ns & \ns & \ns\\
                pctWFarmSelf & \sel & \sel & \sel & \sel & \sel & \sel & \sel & \sel & \sel & \ac & \ac & \ac & \ns & \ns & \ns & \ns & \ns & \ns\\
                pctWInvInc & \sel & \sel & \sel & \ac & \ac & \ac & \ac & \ac & \ac & \ac & \ac & \ac & \ns & \ns & \ns & \ns & \ns & \ns\\
                pctWPubAsst & \sel & \ns & \ns & \ns & \ns & \ns & \ac & \ns & \ns & \ns & \ns & \ns & \ac & \ac & \ac & \ac & \ac & \ac\\
                blackPerCap & \sel & \ns & \ns & \ns & \ns & \ns & \ac & \ns & \ns & \ns & \ns & \ns & \ns & \ns & \ns & \ns & \ns & \ns\\
                indianPerCap & \ns & \sel & \sel & \sel & \ns & \sel & \ns & \sel & \sel & \sel & \ns & \ac & \ns & \ns & \ns & \ns & \ns & \ns\\
                HispPerCap & \sel & \sel & \ns & \ns & \ns & \ns & \ac & \ac & \ns & \ns & \ns & \ns & \ns & \ns & \ns & \ns & \ns & \ns\\
                PctPopUnderPov & \ns & \sel & \ns & \ns & \ns & \ns & \ns & \ac & \ns & \ns & \ns & \ns & \ac & \ac & \ac & \ac & \ac & \ac\\
                PctEmploy & \sel & \ns & \ns & \ns & \ns & \ns & \ac & \ns & \ns & \ns & \ns & \ns & \ns & \ns & \ns & \ns & \ns & \ns\\
                PctEmplProfServ & \ns & \ns & \ns & \sel & \sel & \sel & \ns & \ns & \ns & \ac & \ac & \ac & \ns & \ns & \ns & \ns & \ns & \ns\\
                FemalePctDiv & \sel & \ac & \sel & \ac & \ac & \ac & \ac & \ac & \ac & \ac & \ac & \ac & \ns & \ns & \ns & \ns & \ns & \ns\\
                PctFam2Par & \ns & \ns & \ns & \ns & \ns & \ns & \ns & \ns & \ns & \ns & \ns & \ns & \ac & \ac & \ac & \ac & \ac & \ac\\
                PctKids2Par & \sel & \sel & \sel & \ac & \ac & \ac & \ac & \ac & \ac & \ac & \ac & \ac & \ac & \ac & \ac & \ac & \ac & \ac\\
                PctYoung2Par & \ns & \ns & \ns & \sel & \ac & \sel & \ns & \ns & \ns & \ac & \ac & \ac & \ac & \ac & \ac & \ac & \ac & \ac\\
                PctTeen2Par & \ns & \ns & \ns & \ns & \ns & \ns & \ns & \ns & \ns & \ns & \ns & \ns & \ac & \ac & \ac & \ac & \ac & \ac\\
                PctWorkMom & \sel & \sel & \ac & \ns & \ns & \sel & \ac & \ac & \ac & \ns & \ns & \ac & \ns & \ns & \ns & \ns & \ns & \ns\\
                PctIlleg & \sel & \sel & \sel & \ac & \ac & \ac & \ac & \ac & \ac & \ac & \ac & \ac & \ac & \ac & \ac & \ac & \ac & \ac\\
                PctImmigRec5 & \ns & \sel & \ns & \ns & \ns & \ns & \ns & \ac & \ns & \ns & \ns & \ns & \ns & \ns & \ns & \ns & \ns & \ns\\
                PctImmigRec10 & \sel & \sel & \sel & \ns & \ns & \ns & \ac & \ac & \ac & \ns & \ns & \ns & \ns & \ns & \ns & \ns & \ns & \ns\\
                PersPerRentOccHous & \sel & \ns & \ns & \ns & \ns & \ns & \ac & \ns & \ns & \ns & \ns & \ns & \ns & \ns & \ns & \ns & \ns & \ns\\
                PctPersOwnOccup & \sel & \sel & \sel & \ns & \ns & \ns & \ac & \ac & \ac & \ns & \ns & \ns & \ns & \ns & \ns & \ns & \ns & \ns\\
                PctPersDenseHous & \sel & \sel & \sel & \ac & \ac & \sel & \ac & \ac & \ac & \ac & \ac & \ac & \ns & \ns & \ns & \ns & \ns & \ns\\
                PctHousLess3BR & \sel & \sel & \sel & \sel & \ac & \ac & \ac & \ac & \ac & \ac & \ac & \ac & \ns & \ns & \ns & \ns & \ns & \ns\\
                MedNumBR & \sel & \sel & \ac & \ac & \ac & \sel & \ac & \ac & \ac & \ac & \ac & \ac & \ns & \ns & \ns & \ns & \ns & \ns\\
                HousVacant & \ns & \ns & \sel & \ac & \ac & \ac & \ns & \ns & \ac & \ac & \ac & \ac & \ns & \ns & \ns & \ns & \ns & \ns\\
                PctHousOccup & \ns & \sel & \sel & \sel & \sel & \sel & \ns & \ac & \ac & \ac & \ac & \ac & \ns & \ns & \ns & \ns & \ns & \ns\\
                PctVacantBoarded & \ns & \sel & \sel & \ns & \ns & \ns & \ns & \ac & \ac & \ns & \ns & \ns & \ns & \ns & \ns & \ns & \ns & \ns\\
                PctVacMore6Mos & \sel & \sel & \ns & \ns & \ns & \ns & \sel & \sel & \ns & \ns & \ns & \ns & \ns & \ns & \ns & \ns & \ns & \ns\\
                PctHousNoPhone & \ns & \ns & \ns & \ns & \ns & \ns & \ns & \ns & \ns & \ns & \ns & \ns & \ac & \ac & \ac & \ac & \ac & \ac\\
                MedRentPctHousInc & \ns & \sel & \sel & \sel & \sel & \ac & \ns & \ac & \ac & \ac & \ac & \ac & \ns & \ns & \ns & \ns & \ns & \ns\\
                MedOwnCostPctIncNoMtg & \sel & \sel & \ns & \ns & \ns & \ns & \sel & \sel & \ns & \ns & \ns & \ns & \ns & \ns & \ns & \ns & \ns & \ns\\
                NumStreet & \sel & \ns & \ns & \ns & \ns & \ns & \sel & \ns & \ns & \ns & \ns & \ns & \ns & \ns & \ns & \ns & \ns & \ns\\
                PctBornSameState & \sel & \ns & \ns & \ns & \ns & \ns & \sel & \ns & \ns & \ns & \ns & \ns & \ns & \ns & \ns & \ns & \ns & \ns\\
                PctSameHouse85 & \sel & \sel & \ac & \ns & \sel & \sel & \sel & \ac & \ac & \ns & \ac & \ac & \ns & \ns & \ns & \ns & \ns & \ns\\
                PctSameCity85 & \ns & \ns & \ns & \sel & \ns & \ns & \ns & \ns & \ns & \sel & \ns & \ns & \ns & \ns & \ns & \ns & \ns & \ns\\
                PopDens & \ns & \ns & \ns & \ns & \ns & \sel & \ns & \ns & \ns & \ns & \ns & \ac & \ns & \ns & \ns & \ns & \ns & \ns\\
                LemasPctOfficDrugUn & \ns & \sel & \sel & \ac & \ac & \ac & \ns & \ac & \ac & \ac & \ac & \ac & \ns & \ns & \ns & \ns & \ns & \ns\\

                \hline
        \end{tabular}
    \end{sc}
    \end{small}
        \end{center}
        \vskip -0.1in
    \end{table}
\end{document}